\documentclass[reqno,final,11pt]{amsart}
\usepackage{amssymb,amsthm,amsmath,fullpage,url}
\usepackage[notref,notcite,color]{showkeys}
\usepackage{cite}
\usepackage[dvips]{graphicx}
\usepackage{amsfonts}
\usepackage{stmaryrd}

\usepackage{lscape,array}
\numberwithin{equation}{section}
\usepackage{textcomp}
\usepackage{amsbsy}

\newtheorem{theorem}{Theorem}[section]
\newtheorem{subtheorem}{Theorem}[theorem]
\newtheorem{lemma}[theorem]{Lemma}
\newtheorem{obs}[theorem]{Observation}
\newtheorem{proposition}[theorem]{Proposition}
\newtheorem{cor}[theorem]{Corollary}
\newtheorem{conjecture}[theorem]{Conjecture}
\theoremstyle{remark}
\newtheorem{claim}[subtheorem]{Claim}
\newtheorem{remark}[theorem]{Remark}

\newcommand{\probability}{\mathbf{P}}
\newcommand{\expectation}{\mathbf{E}}
\newcommand{\outneighbor}{\Gamma^+}

\newcommand{\Hom}{\mathrm{Hom}}

\newcommand{\Image}{\mathrm{Im}}
\newcommand{\calF}{\mathcal{F}}
\newcommand{\send}{\eta}
\newcommand{\Psione}{\boldsymbol{\Psi^\lambda}}
\newcommand{\Fork}{\mathcal{\chi}}
\def\lsem{\left\llbracket}
\def\rsem{\right\rrbracket}
\def\Ind{{\rm Ind}}
\def\FFork{{\rm Fork}}

\begin{document}

\title{Counting flags in triangle-free digraphs}

\author{Jan Hladk\'y}
  \address{Institute of Mathematics of the Academy of Sciences of the Czech Republic, \v Zitn\'a 25, Praha. Mathematics Institute is supported by RVO:67985840.}
  \email{honzahladky@gmail.com}
\author{Daniel Kr\'al'}
\address{Mathematics Institute and Department of Computer Science, University of
Warwick, Coventry CV4 7AL, United Kingdom.}
\email{D.Kral@warwick.ac.uk}
\author{Sergey Norin}
\address{Department of Mathematics \& Statistics, McGill
         University, Burnside Hall, 805~Sherbrooke West,
		 Montreal, QC, H3A~2K6, Canada.}
\email{snorin@math.mcgill.ca}
  \thanks{JH and DK were supported partially supported by
  Grant Agency of Charles University, grant
 GAUK~202-10/258009. SN was
  supported in part by NSF under Grant No. DMS-0701033 and an NSERC discovery grant.
  An extended
abstract containing this result appeared in the proceedings of the EuroComb~2009
conference. A major revision of the paper was done during a visit of the first two authors to the Institut Mittag-Leffler (Djursholm, Sweden).
  }

\begin{abstract}
Motivated by the Caccetta--H\"aggkvist Conjecture, we prove
that every digraph on $n$ vertices with minimum outdegree $0.3465n$ contains an oriented triangle. This
improves the bound  of $0.3532n$ of Hamburger, Haxell and
Kostochka.  The main new tool we use in our proof is the theory of
flag  algebras developed recently by Razborov.
\end{abstract}

\maketitle
\section{Introduction}\label{intro}
One of the most intriguing problems of extremal (di)graph
theory is the following conjecture of Caccetta and
H\"aggkvist~\cite{cite:CH} dating back to~1978 (we give definitions used throughout the paper in Section~\ref{sec:FlagAlgebras}).
\begin{conjecture}\label{conj:CH}
Every $n$-vertex digraph with minimum outdegree at least
$r$ has a cycle with length at most $\lceil n/r\rceil$.
\end{conjecture}
For each $r$ and $n$, there is a whole family of digraphs that are believed to be extremal for the conjecture, see~\cite{Razb:CH}. (And the diversity of these digraphs is probably the reason for the difficulty of the conjecture.)
Many results related to the conjecture can be found in a survey by Sullivan~\cite{Sullivan}.

The case when $r=n/3$ is of particular interest. It
asserts that any $n$-vertex digraph with minimum
outdegree at least $n/3$ contains a triangle.
Our main result gives a new minimum outdegree bound for
this case of  the Caccetta--H\"aggkvist Conjecture.
\begin{theorem}\label{thm:main}
Every $n$-vertex digraph with minimum outdegree  at least
$0.3465n$ contains a triangle.
\end{theorem}
Our result improves the previous known minimum-degree bounds established by
Caccetta and H\"aggkvist~\cite{cite:CH} ($0.3820n$),
Bondy~\cite{cite:Bondy} ($0.3798n$), Shen~\cite{Shen}
($0.3543n$) and Hamburger, Haxell, and Kostochka~\cite{HHK} ($0.3532n$).

The proof of Theorem~\ref{thm:main} uses the framework of flag algebras
which was developed by Razborov~\cite{FlagAlgebras}.
This framework provides a general formalism
which allows to deal with problems in
extremal combinatorics. Razborov used this approach to
solve a long-standing open problem on density of
triangles in graphs~\cite{RazborovTriangles}, and a special case of the
Tur\'an's problem for 3-uniform
hypergraphs~\cite{RazborovHyper}.
After posting the first version of this manuscript, several other applications of flag algebras appeared, see e.g.~\cite{BabTal:Flags,Grzesik:Pentagon,HHKNR:Pentagon,HHKNR:Wheel,Sudakov:Flags,Pikhurko:Flag,23:Flag,MonoTr:Flag}. In particular, Razborov~\cite{Razb:CH} proved Conjecture~\ref{conj:CH} with $r=n/3$ for digraphs avoiding three specific digraphs on four vertices. In addition, a software package that can be used to apply flag algebra methods to extremal combinatorics was developed by Vaughan. The package is publicly available at {\tt http://www.maths.qmul.ac.uk/$\sim$ev/flagmatic/}.

There are two more ingredients that we use in addition to the standard use of flag algebras,
which is also referred to as the ``semidefinite method'' by Razborov.
One of them is a variant of inductive arguments which can be found in~\cite{FlagAlgebras} and
the other is a result of Chudnovsky, Seymour and Sullivan~\cite{ChuSeySull}
on eliminating cycles in triangle-free digraphs.
A brute force computer search was used to combine these ingredients to give the bound.
However, the resulting proof is close to being computer-free,
only with Maple used to verify several hundred addition
and multiplication operations involving five-to-nine digit numbers.

The paper is organized as follows. In
Section~\ref{sec:FlagAlgebras} we present the notation used in the paper and
we survey the framework of flag algebras as needed in our proof.
The structure of triangle-free digraphs is treated in Section~\ref{sec:Struct}.
It contains a statement of the key Theorem~\ref{prop:asymptotic} and
gives a short proof of Theorem~\ref{thm:main} based on it.
Finally, we give a proof of Theorem~\ref{prop:asymptotic} in Section~\ref{sec:ProofofProp}.

\section{Notation}\label{sec:FlagAlgebras}

We start with introducing a general notation related to directed graphs.
A \emph{digraph} is a directed graph with no loops, no
parallel edges, and no counter-parallel edges.
A \emph{subdigraph} of a digraph $D$ is a digraph that can be obtained from $D$ by deleting some of its vertices and edges.
Given a digraph $D$ and a set $U\subseteq V(D)$ we write $D\setminus U$ for the subdigraph of $D$ obtained by deleting the vertices of $U$ and
$D[U]$ for the subdigraph induced by $U$, i.e. the subdigraph obtained by deleting all vertices except for those in $U$.
A \emph{cycle} of length $t$ is a digraph $C_t$ with $t$ vertices $v_0,\ldots, v_{t-1}$ and $t$ edges $v_iv_{i+1}$ (indices modulo $t$).
A \emph{triangle} is a cycle of length three. Finally, a digraph is \emph{acyclic} if it does not contain any cycle as a subdigraph.

If $D$ is a digraph, we write $V(D)$ and $E(D)$ for the set of vertices and for the set of edges of a digraph $D$.
A vertex $v$ is an \emph{outneighbor} of $u$ if $D$ contains an edge $uv$.
The \emph{outneighborhood}
$\outneighbor(u)$ of a vertex $u\in V(D)$ is the set of all outneighbors of $u$,
i.e. $\outneighbor(u)=\{v\in V(D)\::\: uv\in E(D)\}$.
The \emph{outdegree} of a vertex $u$, denoted by $\deg^+(u)$, is the number of its outneighbors, i.e. $\deg^+(u)=|\outneighbor(u)|$.
For a set $U\subseteq V(D)$,
the \emph{common outneighborhood} of $U$ is the set of the common outneighbors of the vertices contained in $U$,
i.e.  $\outneighbor(U)=\bigcap_{u\in U}\outneighbor(u)$.
A digraph $D$ is \emph{outregular}
if all vertices of $D$ have the same outdegree.
If $D$ is a digraph, we write $\delta^+(D)$ for the minimum outdegree of a vertex of $D$, i.e. $\delta^+(D)=\min_{u\in V(D)}\deg^+(u)$.

\subsection{Flags}
A thorough introduction to flag algebras can be found in~\cite{FlagAlgebras}.
Here, we present those concepts needed in our proof of Theorem~\ref{thm:main}.
We follow the notation as used in~\cite{FlagAlgebras} but
we restrict our presentation to the case of triangle-free digraphs.
So, we will be dealing (using the language from~\cite{FlagAlgebras})
with the theory $\mathcal T$ of triangle-free digraphs, which is
a vertex-uniform theory with amalgamation property.
The former means that there exists a unique (up to isomorphism) one-vertex digraph and
the latter represents the fact that union of two triangle-free digraphs is a triangle-free digraph.

We will introduce an algebra with addition and multiplication
on formal linear combinations of unrooted and rooted digraphs,
which we will refer to as flags.
In the case of rooted digraphs, we want to refer to the subgraph
induced by the roots as a type. Formally,
a \emph{type} of order $k$ is a triangle-free $k$-vertex
digraph $\sigma$ on the vertex set $V(\sigma)=[k]$.
We will write $|\sigma|$ for the order of $\sigma$, i.e. $|\sigma|=k$.
A \emph{$\sigma$-flag} is a pair $F=(D,\theta)$
where $D$ is a triangle-free digraph and $\theta:[k]\rightarrow
V(D)$ is an isomorphism of $\sigma$ and $D[\Image(\theta)]$.
A particular example of a $\sigma$-flag is the flag comprised of $\sigma$ and the identity mapping on $[k$];
slightly abusing the notation, we will use $\sigma$ for this $\sigma$-flag.
Since we think of $\sigma$-flags as rooted at $\sigma$,
we sometimes refer to the vertices of $\Image(\theta)$ as to the roots.

We now define two different notions of a restriction: a restriction of
a $\sigma$-flag and a restriction of a type.
A \emph{restriction} of a $\sigma$-flag $F=(D,\theta)$ to a set $U\subseteq V(D)$ such that $\Image(\theta)\subseteq U$
is the $\sigma$-flag $(D[U],\theta)$ which will be denoted by $F|_{U}$.
A \emph{restriction} of a type $\sigma$ of order $k$ for an injective map $\eta:[k']\to [k]$
is the type $\sigma_{\eta}$ with vertex set $[k']$ and with $ij$ being an edge iff $\eta(i)\eta(j)$ is an edge in $\sigma$.
In particular, $(\sigma,\eta)$ is a $\sigma_{\eta}$-flag of order $|\sigma|$.

Suppose that $\sigma$ is a type of order $k$.
We write $\mathcal{F}^\sigma$ for the set of all
$\sigma$-flags and $\mathcal{F}^\sigma_\ell$ for those of order $\ell$.
We will consider two $\sigma$-flags $F_1=(D_1,\theta_1)$ and $F_2=(D_2,\theta_2)$ to be \emph{isomorphic}
if there exists their isomorphism $f:V(D_1)\to V(D_2)$ that is an identity on $\sigma$,
i.e., the restriction of $f$ to $\Image(\theta_1)$ is $\theta_2\circ\theta_1^{-1}$.
If two $\sigma$-flags $F_1$ and $F_2$ are isomorphic, we write $F_1\cong_\sigma F_2$.
As a slight extension of this notation,
we will use $\mathcal{F}$ for the set of all digraphs,
$\mathcal{F}_\ell$ for the set of all digraphs of order $\ell$, and
$F_1\cong F_2$ to denote that $F_1$ and $F_2$ are isomorphic.
To ease our way of expressing,
we will also think of $\mathcal{F}$ and $\mathcal{F}_\ell$ as
of $\mathcal{F}^\sigma$ and $\mathcal{F}^\sigma_\ell$ for the empty type $\sigma$.

\subsection{Frequently used flags}

We now introduce notation for the most frequently used flags.
The notation is illustrated in Figure~\ref{fig:wp}.
For depicting digraphs (also see Figure~\ref{fig:nonedges} for illustration),
we use solid line arrows to show oriented edges, and dashed lines to depict their absence.
When two vertices are not connected by an arrow or a dashed line in a figure,
the pair is connected with a grey solid line.
This represents that the pair should be expanded into a formal sum of three flags (non-edge and the two orientations of an edge).
A dashed arrow in a figure represents a formal sum of two flags with a non-edge and an edge in the opposite direction of the arrow.

\begin{figure}
  \centering
\includegraphics[scale=0.92]{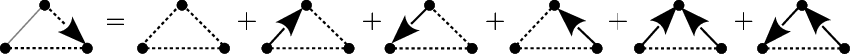}
\caption{Example of usage of dashed arrows and grey solid lines.}
  \label{fig:nonedges}
\end{figure}

The symbol $\lambda$ denotes the unique type of order one.
As explained earlier, we will also use $\lambda$ for the $\lambda$-flag $(\lambda,{\rm id})$.
The digraph consisting of a single directed edge is $\varrho$.
The $\lambda$-flags obtained from $\varrho$ by labelling the tail and the head
are denoted by $\alpha$ and $\bar\alpha$, respectively.
The type consisting of a single directed edge is $\beta$.

The $\lambda$-flag consisting of two vertices, one of them being the root, is $\gamma$.
The \emph{fork}, which is denoted by $\kappa$, is the digraph that consists of three vertices $a,b,c$ and two edges $ab$ and $ac$.
The vertex $a$ is called the \emph{center} of $\kappa$.
When the fork is rooted at its center, it becomes a $\lambda$-flag denoted by $\Fork$.

\begin{figure}
  \centering
\includegraphics[scale=0.92]{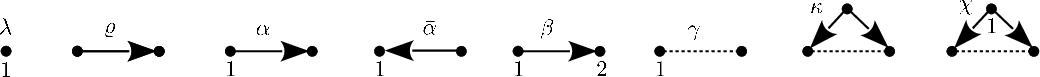}
\caption{Frequently used types and flags.}
  \label{fig:wp}
\end{figure}

\subsection{Flag algebras}
We shall now enhance $\calF^\sigma$ with the structure of an algebra.
The motivation for the definitions now presented becomes clear in the next subsection
where we introduce the convergence of $\sigma$-flags.
To give at least a partial motivation for the definitions we now present,
let us say that the structure we define on $\calF^{\sigma}$
should behave consistently with the probabilities of seeing the $\sigma$-flags involved in a large $\sigma$-flag.

If $F'$ and $F$ are two digraphs of orders $\ell'\le\ell$, respectively, then
$$p(F';F)=\probability\left[F[{\mathbf U}]\cong F'\right]$$ where $\mathbf U$ is a random $\ell'$-element subset of $V(F)$.
Note that we use bold letters to denote random objects following the notation used e.g. in~\cite{FlagAlgebras}.
The definition can be extended to $\sigma$-flags by picking a random subset of non-root vertices.
Formally, 
if $F=(D,\theta)\in \calF^\sigma_\ell$ and $F'\in\calF^\sigma_{\ell'}$ are two $\sigma$-flags, $\ell'\le \ell$,
we define the quantity $p(F';F)$ by
$$p(F';F)=\probability\left[F|_{\Image(\theta)\cup\mathbf V}\cong_\sigma F'\right]\;,$$
where $\mathbf V$ is an
$(\ell'-|\sigma|)$-element subset of $V(D) \setminus \Image(\theta)$ taken uniformly at random.
Note that this is consistent with viewing $\mathcal{F}$ and $\mathcal{F}_\ell$ as
$\mathcal{F}^\sigma$ and $\mathcal{F}^\sigma_\ell$ with $\sigma$ being the empty type.
For completeness, we define $p(F';F)$ to be zero if the order of $F'$ is larger than the order of $F$.
This allows us to view the values $p(F';F)$ for a fixed $\sigma$-flag $F$ as a vector indexed by $\mathcal F^\sigma$:
so, we define $p^F$ to be the vector from $[0,1]^{\mathcal F^\sigma}$ such that $p^F_{F'}=p(F';F)$.

The following chain rule follows directly from the definition (cf.~\cite[Lemma~2.2]{FlagAlgebras}).
\begin{lemma}\label{lem:chainrule}
Let $\sigma$ be a (possibly empty) type and
let $\ell'\le\tilde\ell\le\ell$, $F\in\mathcal F^\sigma_\ell$ and $F'\in\mathcal F^\sigma_{\ell'}$.
It holds that
$$p(F';F)=\sum_{\tilde F\in\mathcal F^\sigma_{\tilde\ell}}p(F';\tilde F)p(\tilde F;F)\;.$$
\end{lemma}
Informally speaking, Lemma~\ref{lem:chainrule} says that the probabilities of seeing a $\sigma$-flag of order $\ell'$
can be computed from those of seeing $\sigma$-flag of order $\tilde\ell$ for some $\tilde\ell>\ell'$.
This leads as to the definition of an algebra $\mathcal A^\sigma$ that follows.
We consider the $\mathcal K^\sigma$ of the space $\mathbb R\mathcal F^\sigma$ of finite
formal linear combinations of $\sigma$-flags that is generated by the combinations of the form
\begin{equation}\label{eq:factor}F'-\sum_{\tilde
F\in\mathcal F^\sigma_{\tilde \ell}}p(F';\tilde F)\tilde
F\;,
\end{equation}
for all $\sigma$-flag $F'\in \mathcal F^\sigma_{\ell'}$ and all pairs $\ell'$ and $\tilde\ell$ such that $\ell'\le\tilde\ell$.
We then set $\mathcal A^\sigma=\mathbb R\mathcal F^\sigma/\mathcal K^\sigma$.
The factor-space $\mathcal A^\sigma$ inherits the additive structure from $\mathbb R\mathcal F^\sigma$.
In what follows,
we identify the elements of $\mathbb R\mathcal F^\sigma$ with their classes in $\mathcal A^\sigma$,
i.e.~when we speak about the $\sigma$-flag $F$ as an element of $\mathcal A^\sigma$,
we mean the class $F+\mathcal K^\sigma$.

We now aim at defining the product operation on $\mathcal A^\sigma$.
The motivation again comes from the definition of convergence given in the next subsection.
If $F_1\in\mathcal F^\sigma_{\ell_1}, F_2\in\mathcal F^\sigma_{\ell_2}$ and $F\in\mathcal F^\sigma_{\ell}$, $\ell\ge \ell_1+\ell_2-|\sigma|$
are three $\sigma$-flags, we define the quantity $p(F_1,F_2;F)$ to be
$$p(F_1,F_2;F)=\probability\left[F|_{\Image(\theta)\cup\mathbf V_1}\cong_\sigma F_1\mbox{ and }F|_{\Image(\theta)\cup\mathbf V_2}\cong_\sigma F_1\right]\;.$$
where $(\mathbf V_1, \mathbf V_2)$ is a pair of
disjoint subsets of $V(D) \setminus \Image(\theta)$ of cardinalities $\ell_1-|\sigma|$ and
$\ell_2-|\sigma|$, respectively, drawn uniformly at random from the space of all such pairs.
This definition allows us to define a bilinear mapping 
$\cdot\:: \mathcal F^\sigma\otimes\mathcal F^\sigma\to\mathbb R\mathcal F^\sigma$ as
$$F_1\cdot F_2=\sum_{F\in\mathcal F^\sigma_\ell}p(F_1,F_2;F)F$$
where $F_1\in\mathcal F^\sigma_{\ell_1}, F_2\in\mathcal F^\sigma_{\ell_2}$ and $F\in\mathcal F^\sigma_{\ell}$, $\ell\ge \ell_1+\ell_2-|\sigma|$.
The mapping $\cdot$ can be extended by linearity to $\mathbb R\mathcal F^\sigma\otimes\mathbb R\mathcal F^\sigma$.
It can be shown~\cite{FlagAlgebras} that $\mathcal K^\sigma$ defines a congruence with respect to this mapping and
the mapping $\cdot$ gives a well-defined multiplication operation in $\mathcal A^\sigma=\mathbb R\mathcal F^\sigma/\mathcal K^\sigma$.
The unit element with respect to the multiplication is the $\sigma$-flag $\sigma$ (recall that we identify the elements
$\mathbb R\mathcal F^\sigma$ with their classes in $\mathcal A^\sigma$).

\subsection{Convergence}
The notions presented in this subsection provide motivation for the definitions we have introduced earlier.
Fix a type $\sigma$.
A sequence of $\sigma$-flags $\{F_n\}_{n=1}^\infty$ \emph{converges} to a point $x\in[0,1]^{\mathcal F^\sigma}$
if the sequence $\{p^{F_n}\}_{n=1}^\infty$ converges to $x$ in the product topology on $[0,1]^{\mathcal{F}^\sigma}$.
The vector $x$ gives rise to a mapping $\Psi:\mathcal A^\sigma\to\mathbb R$
defined by $\Psi(F):=x_F$ for $F\in\mathcal{F}^\sigma$ and extended linearly to $\mathcal A^\sigma$.
It can be shown~\cite{FlagAlgebras} that if the orders of $F_n$ grow to infinity,
then the mapping $\Psi$ is a homomorphism from $\mathcal A^\sigma$ to $\mathbb R$.
We then write $\lim\limits_{n\to\infty} F_n=\Psi$.

Let $\Hom(\mathcal A^\sigma,\mathbb R)$ be the set of all algebra homomorphisms from $\mathcal A^\sigma$ to $\mathbb R$ and
let $\Hom^+(\mathcal A^\sigma,\mathbb R)\subseteq \Hom(\mathcal A^\sigma,\mathbb R)$ be those
homomorphisms $\Psi$ such that $\Psi(F)\ge 0$ for every $F\in\mathcal F^\sigma$.
Note that the homomorphism $\Psi$ defined in the previous paragraph belongs to $\Hom^+(\mathcal A^\sigma,\mathbb R)$.
This correspondence goes both ways as stated in the next theorem (cf.~\cite[Theorem~2.5]{LovSze06},~\cite[Theorem~3.3]{FlagAlgebras}).
\begin{theorem}\label{thm:limitsexist}
Let $\sigma$ be a type.
For every $\Psi\in \Hom^+(\mathcal A^\sigma,\mathbb R)$,
there exists a sequence of $\sigma$-flags $\{F_n\}_{n=1}^\infty$ with growing orders that converges and
$\lim\limits_{n\to\infty} F_n=\Psi$.

On the other hand, if $\{F_n\}_{n=1}^\infty$ is a sequence of $\sigma$-flags with orders growing to infinity,
then there exists a subsequence $\{F_{n_i}\}_{i=1}^\infty$ of the sequence $\{F_n\}_{n=1}^\infty$ that
converges and $\lim\limits_{n\to\infty} F_{n_i}\in\Hom^+(\mathcal A^\sigma,\mathbb R)$.
\end{theorem}
Note that a particular corollary of Theorem~\ref{thm:limitsexist} is that
$\Psi(F)\in [0,1]$ for every $\Psi\in\Hom^+(\mathcal A^\sigma,\mathbb R)$ and every $F\in\mathcal F^\sigma$.

We now aim to define a partial order $\le_\sigma$ on $\mathcal A^\sigma$ to compare ``densities'' in convergent sequences of $\sigma$-flags.
If $a,b\in\mathcal A^\sigma$, then $a\le_\sigma b$ iff
$\Psi(a)\le\Psi(b)$ for every $\Psi\in\Hom^+(\mathcal A^\sigma,\mathbb R)$.
Observe that
if $a,b\in\mathcal A^\sigma$ are such that
$b-a=\sum_{F\in\mathcal F^\sigma}c_F F$ with all $c_F\in\mathbb R$ being nonnegative, then $a\le_\sigma b$.

\subsection{Random homomorphisms and averaging}
In the previous subsection,
we have associated every convergent sequence $(D_n)_{n=1}^\infty$ of triangle-free digraphs
with a homomorphism $\Psi\in\Hom^+(\mathcal A,\mathbb R)$.
We now associate it with a probability distribution $\probability^{\sigma}$
on homomorphisms from $\Hom^+(\mathcal A^\sigma,\mathbb R)$ for non-empty types $\sigma$.
Fix a type $\sigma$ of order $k$ such that $\Psi(\sigma)>0$.
Every mapping $\theta:[k]\to V(D_n)$ such that $\theta$ is an isomorphism from $\sigma$ to $D_n[\Image(\theta)]$
yields a $\sigma$-flag, which is $(D_n,\theta)$, and
it consequently leads to a mapping from $\mathcal A^\sigma$ to $\mathbb R$, which is $p^{(D_n,\theta)}$.
By choosing the mapping $\theta$ uniformly at random among all injective mappings from $[k]$ to $V(D_n)$ such that
$\theta$ is an isomorphism from $\sigma$ to $D_n[\Image(\theta)]$,
we obtain a probability distribution on mappings from $\mathcal A^\sigma$ to $\mathbb R$.
Note that we obtain one probability distribution on mappings from $\mathcal A^\sigma$ to $\mathbb R$ for each $n\in\mathbb N$.
It can be shown (for the natural notion of convergence) that these probability distributions on mappings
from $\mathcal A^\sigma$ to $\mathbb R$ converge to a probability distribution $\probability^{\sigma}$ on $\Hom^+(\mathcal A^\sigma,\mathbb R)$.
The rest of this subsection is devoted to formalizing the connection between
the homomorphism $\Psi\in\Hom^+(\mathcal A,\mathbb R)$ and the distributions $\probability^{\sigma}$.

Fix a type $\sigma$ of order $k$ and its restriction $\sigma_0=\sigma|_{\eta}$ of order $k'$ (recall that $\eta:[k']\to[k]$).
The \emph{unlabelling} of a $\sigma$-flag $F=(D,\theta)$ is the $\sigma_0$-flag $F|_{\eta}=(D,\theta\circ\eta)$.
Let $\boldsymbol{\theta}':[k]\rightarrow V(D)$ be an injective extension of
the map $\theta\circ\eta:[k']\rightarrow V(D)$ taken uniformly at random among all such injective extensions.
The quantity $q_{\sigma,\eta}(F)$ is the probability that $(D,\boldsymbol{\theta}')$ and $F$ are isomorphic $\sigma$-flags.
The \emph{averaging operator} $\lsem \rsem_{\sigma,\eta}:\mathcal A^\sigma\rightarrow \mathcal A^{\sigma_0}$ is the linear extension
of the map defined on $\mathcal F^\sigma$ as
$$\lsem F\rsem_{\sigma,\eta}=q_{\sigma,\eta}(F)F|_{\eta}\;.$$
When $\eta$ is the null mapping, i.e. $k'=0$, we write  $\lsem \rsem_{\sigma}$ instead of $\lsem \rsem_{\sigma,\eta}$ for brevity.
In the case of $k'=1$, we also write $\lsem \rsem_{\sigma,m}$ instead of $\lsem \rsem_{\sigma,\eta}$ where $m=\eta(1)$.

We further develop the correspondence from the first paragraph of this subsection,
which corresponds to the arguments given below for $\eta$ being the null mapping.
Recall that a type $\sigma$ of order $k$ and its restriction $\sigma_0=\sigma|_{\eta}$ of order $k'$ are fixed.
For a homomorphism $\Psi\in\Hom^{+}(\mathcal A^{\sigma_0},\mathbb R)$ with $\Psi((\sigma,\eta))>0$,
we say that the probability distribution $\probability^{\sigma,\eta}$ on sets of $\Hom^+(\mathcal A^{\sigma},\mathbb R)$
\emph{extends} the homomorphism $\Psi$ if
$$\int_{\Hom^+(\mathcal
A^{\sigma},\mathbb
R)}\Psi(f)\probability^{\sigma,\eta}(d\Psi)=\frac{\Psi(\lsem f\rsem_{\sigma,\eta})}{\Psi(\lsem
\sigma\rsem_{\sigma,\eta})}$$ for all $f\in\mathcal
A^\sigma$. Theorem 3.5 from~\cite{FlagAlgebras} asserts that an extension always exists and it is unique.
\begin{theorem}
Let $\sigma$ be a type of order $k$ and let $\sigma_0=\sigma|_{\eta}$ be its restriction.
For every homomorphism $\Psi\in\Hom^{+}(\mathcal A^{\sigma_0},\mathbb R)$ with $\Psi((\sigma,\eta))>0$,
there exists a unique probability distribution $\probability^{\sigma,\eta}$ on $\Hom^+(\mathcal A^{\sigma},\mathbb R)$ that extends $\Psi$.
\end{theorem}
If $\Psi\in\Hom^{+}(\mathcal A^{\sigma_0},\mathbb R)$ is fixed,
then the \emph{random homomorphism rooted at $\sigma$} is a random homomorphism given by the unique distribution $\probability^{\sigma,\eta}$
that extends the homomorphism $\Psi$.
The random homomorphism rooted at $\sigma$ is denoted by $\boldsymbol{\Psi}^{\sigma,\eta}$.
It follows from the definition of the extension that
\begin{equation}\label{eq:Razb21}
\expectation[\boldsymbol{\Psi}^{\sigma,\eta}(f)]=\frac{\Psi(\lsem
f\rsem_{\sigma,\eta})}{\Psi(\lsem \sigma\rsem_{\sigma,\eta})}\;,
\end{equation}
for all $f\in\mathcal A^\sigma$.
If $\eta$ is the null mapping, we will often drop it from the superscript.
Using the just introduced terminology,
it can be shown that the distributions $\probability^{\sigma}$ on $\Hom^+(\mathcal A^{\sigma},\mathbb R)$ as defined in the first paragraph
extend the homomorphism $\Psi\in\Hom^+(\mathcal A,\mathbb R)$ associated with the convergent sequence $(D_n)_{n=1}^\infty$ of digraphs.

\subsection{Minimum outdegree}
The notion of a random homomorphism leads to a natural definition of
the \emph{minimum outdegree} $\delta_\alpha(\Psi)$ of a homomorphism $\Psi\in\Hom^+(\mathcal A,\mathbb R)$.
This is defined by
$$\delta_\alpha(\Psi)=\sup\{a\::\:\probability[\boldsymbol{\Psi}^{\lambda}(\alpha)<a]=0\}\;.$$
It is not true that if a sequence $\{D_n\}_{n=1}^\infty$ of digraphs converges to $\Psi$,
then $$\lim_{n \to \infty}
\delta^+(D_n)/|V(D_n)|=\delta_\alpha(\Psi)\;\mbox{.}$$
For example,
if $D_n$ consists of a single isolated vertex and a digraph formed by four sets of $n$ vertices $U_1$, $U_2$, $U_3$ and $U_4$
with edges going from $U_i$ to $U_{i+1}$ (indices modulo four),
then the sequence $\{D_n\}_{n=1}^\infty$ converges,
$\delta^+(D_n)=0$ for every $n$ and $\delta_\alpha(\Psi)=1/4$ for the limit homomorphism $\Psi$.
However, the converse is true:
if all digraphs $D_n$ have large minimum outdegree, then $\delta_\alpha(\Psi)$ for the limit homomorphism $\Psi$ is also large~\cite{FlagAlgebras}.
\begin{theorem}\label{thm:mindegreesequences}
Suppose that $\{D_n\}_{n=1}^\infty$ is a sequence which converges to $\Psi$.
Then $$\delta_\alpha(\Psi)\ge\liminf_{n \to \infty}
\frac{\delta^+(D_n)}{|V(D_n)|}\;.$$
\end{theorem}
As discussed above, the converse of Theorem~\ref{thm:mindegreesequences} need not hold in general.
However, a weaker statement is true:
for every homomorphism $\Psi$ with large minimum outdegree,
there exists a sequence convergent to $\Psi$ with large minimum outdegree.
In fact, a sequence of digraphs $D_n$ where $D_n=F\in\mathcal F_n$ with probability $\Psi(F)$
converges with probability one and it has the desired property with probability one (cf.~\cite[Section~2.6]{LovSze06}).
\begin{theorem}\label{thm:mindegreeconvergence}
For every $\Psi\in\Hom^+(\mathcal A^\sigma,\mathbb R)$,
there exists a sequence $\{D_n\}_{n=1}^\infty$ of digraphs that
converges to $\Psi$, and such that
$$\lim_{n \to \infty}
\frac{\delta^+(D_n)}{|V(D_n)|}=\delta_\alpha(\Psi)\;.$$
\end{theorem}

We now relate the outdegree distribution of a convergent sequence $\{D_n\}_{n=1}^\infty$ of digraphs and
the associated homomorphism $\Psi\in\Hom^+(\mathcal A,\mathbb R)$.
If $D$ is a digraph and $c\in [0,1]$, then
$$S(D,c):=\frac{|\{v\in V(D)\::\:\deg^+(v)\le cn\}|}{n}\;\mbox{.}$$
Recall there exists a unique distribution $\Psione$ on $\Hom^+(\mathcal A^\lambda,\mathbb R)$ that extends $\Psi$.
A consequence of \cite[Theorem~3.12]{FlagAlgebras} is the following.
\begin{lemma}\label{lem:limitoutdegDistr}
Let $\{D_n\}_{n=1}^\infty$ be a convergent sequence of digraphs with $\lim_{n\to\infty} D_n=\Psi$.
It holds that
$$\probability[\boldsymbol{\Psi^{\lambda}}(\alpha)\le c]\ge \liminf_{n\rightarrow\infty} S(D_n,c)\;.$$
\end{lemma}
The inequality in Lemma~\ref{lem:limitoutdegDistr} can be strict:
let $D_n$ be a digraph formed by two sets of $n$ and $n+1$ vertices with all edges going from the smaller set to the larger one.
Then $S(D_n,1/2)<1/2$, but $\probability[\boldsymbol{\Psi^{\lambda}}(\alpha)\le 1/2]=1$.

\subsection{Cauchy--Schwarz inequality}
One of the frequently used tools in extremal combinatorics
is the Cauchy--Schwarz Inequality. Recently, Lov\'asz and
Szegedy~\cite{LovSz09} made progress on
formalizing its importance in the context of extremal graph
theory. They have shown that every linear inequality between
subgraph densities that holds asymptotically for all
graphs can be approximated (with arbitrary
precision) by finitely many applications of the Cauchy--Schwarz Inequality.
However, it might not be possible to prove it exactly as shown by Hatami and the third author~\cite{HN:Undecide}.
The Cauchy--Schwarz Inequality reads in the language of the flag algebras as
follows (cf.~\cite[Theorem~3.14]{FlagAlgebras}).
\begin{theorem}\label{thm:CauchySchwarz}
If $\sigma$ is a type and $\sigma_0=\sigma|_{\eta}$ is one of its restrictions,
it holds that
$$\Psi\left(\lsem f^2\rsem_{\sigma,\eta}\right)\ge 0$$
for all $f\in\mathcal A^\sigma$ and all $\Psi\in \Hom^+(\mathcal A^{\sigma_0},\mathbb R)$.
\end{theorem}
The proof of Theorem~\ref{thm:CauchySchwarz} follows the next lines.
Observe that $\Psi'(f^2)=\Psi'(f)^2\ge 0$ for every $\Psi'\in\Hom^+(\mathcal A^{\sigma},\mathbb R)$ and every $f\in\mathcal A^\sigma$.
If $\Psi((\sigma,\eta))>0$,
then $\boldsymbol{\Psi^{\sigma,\eta}}(f^2)\ge 0$ and so
$\Psi(\lsem f^2\rsem_{\sigma,\eta})=\Psi(\lsem \sigma\rsem_{\sigma,\eta})\expectation[\boldsymbol{\Psi}^{\sigma,\eta}(f^2)]\ge 0$.
If $\Psi((\sigma,\eta))=0$, then $\Psi(\lsem f\rsem_{\sigma,\eta})=0$ for every $f\in F^{\sigma}$ and the statement is trivial.

\subsection{Inductive arguments}
To formalize inductive arguments in the language of flag algebras,
Razborov~\cite{FlagAlgebras} introduces  the notion of an \emph{upward operator}.
We only use two special instances of this operator which we now define.

Let $\sigma$ be a type of order $k$ and let $\sigma_0=\sigma|_{\eta}$ be one of its restrictions of order $k'$.
For a $\sigma$-flag $F=(D,\theta)$, we define
$F\downarrow_{\eta}$ to be the $\sigma_0$-flag obtained from $F$ by deleting the vertices corresponding to $\sigma$ but not to $\sigma_0$,
i.e.,
$$F\downarrow_{\eta}:=F|_{\eta}\setminus\theta([k]\setminus\Image(\eta))\;.$$
The operator $\pi^{\sigma,\eta}: \mathcal{A}^{\sigma_0} \to \mathcal{A}^{\sigma}$ is defined by its action on $\mathcal F^{\sigma|_{\eta}}$
as follows $$\pi^{\sigma,\eta}(F)=\sum_{\substack{\hat{F} \in F^{\sigma} \\ \hat{F}\downarrow_{\eta}=F}} \hat F\;,$$
for $F \in \mathcal{F}^{\sigma_0}$.
The following properties of $\pi^{\sigma,\eta}$ were established in~\cite[Theorem~3.18, Corollary~3.19, Remark~5]{FlagAlgebras}.
\begin{theorem}\label{thm:piproperties}
Let $\sigma_2$ be a type of order $k_2$,
let $\sigma_1=\sigma_2|_{\eta_{21}}$ be one of its restrictions of order $k_1\le k_2$, and
let $\sigma_0=\sigma_1|_{\eta_{10}}$ be one of the restrictions of $\sigma_1$ of order $k_0\le k_1$,
i.e. $\sigma_0=\sigma_2|_{\eta_{20}}$ where $\eta_{20}=\eta_{21}\circ\eta_{10}$.
Suppose that $\Psi \in \Hom(\mathcal{A}^{\sigma_0},\mathbb{R})$ is a homomorphism such that $\Psi((\sigma_2, \eta_{20})) > 0$.
\begin{enumerate}
  \item[a)] For every $f \in \mathcal{A}^{\sigma_0}$ we have
  $$\probability[\boldsymbol{\Psi^{\sigma_1,\:\eta_{10}}}(\pi^{\sigma_1,\eta_{10}}(f))=\Psi(f)]=1\;.$$
  \item[b)] For every $f \in \mathcal{A}^{\sigma_1}$ we have
$$\probability[\boldsymbol{\Psi^{\sigma_1,\:\eta_{10}}}(f)=0]=1
\quad\Rightarrow\quad
\probability[\boldsymbol{\Psi^{\sigma_2,\:\eta_{20}}}(\pi^{\sigma_2,\eta_{21}}(f))=0]=1\;.$$
\end{enumerate}
\end{theorem}
Note that the assumption $\Psi((\sigma_2, \eta_{20})) > 0$ in Theorem~\ref{thm:piproperties} implies that $\Psi((\sigma_1, \eta_{10})) > 0$.

The second operator we define we refer to as the \emph{replication operator}.
Let $\sigma$ be a type of order $k$,
let $\eta:[k-1] \to [k]$ be an injective mapping, and
let $\sigma'$ be a flag of order $k'$.
Further, let $i_0$ be the unique integer contained in $[k]\setminus\Image(\eta)$.
We will define a flag $\sigma'\wr(\sigma,\eta)$ of order $k+k'-1$.
The flag $\sigma'\wr(\sigma,\eta)$ is the unique digraph $D$ with vertex-set $[k+k'-1]$ such that
$ij$ is an edge of $D$ if one of the following holds:
\begin{itemize}
\item $i\le k-1$, $j\le k-1$ and $\eta(i)\eta(j)$ is an edge of $\sigma$,
\item $i\le k-1$, $j\ge k$ and $\eta(i)i_0$ is an edge of $\sigma$,
\item $i\ge k$, $j\le k-1$ and $i_0\eta(j)$ is an edge of $\sigma$, or
\item $i\ge k$, $j\ge k$ and $(i-k+1)(j-k+1)$ is an edge of $\sigma'$.
\end{itemize}
In other words, the vertices $1,\ldots,k-1$ of $\sigma'\wr(\sigma,\eta)$ induce a digraph isomorphic to $\sigma$ restricted to $\Image(\eta)$,
the other $k'$ vertices of $\sigma'\wr(\sigma,\eta)$ are joined to the first $k-1$ vertices as the vertex $i_0$ to the rest of $\sigma$, and
the vertices $k,\ldots,k+k'-1$ induce a digraph isomorphic to $\sigma'$.
The construction is illustrated in Figure~\ref{fig:OPERATION}.
\begin{figure}
  \centering
\includegraphics[scale=1]{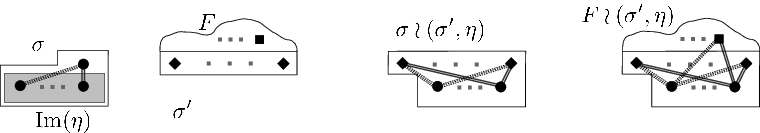}
\caption{The replication operator $\wr$. From the types $\sigma$ and $\sigma'$ and the $\sigma'$-flag $F$ on the left, we get the type $\sigma\wr(\sigma',\eta)$ and the $\sigma\wr(\sigma',\eta)$-flag $F\wr(\sigma',\eta)$ on the right. Two types of lines represent possibly different connection types (non-edges or edges in either direction) in $\sigma$, which are then duplicated in $\sigma\wr(\sigma',\eta)$ and $F\wr(\sigma',\eta)$.}
  \label{fig:OPERATION}
\end{figure}

The definition naturally extends to $\sigma'$-flags.
If $F'$ is a $\sigma'$-flag,
then the $\sigma'\wr(\sigma,\eta)$-flag $F'\wr(\sigma,\eta)$ is the unique $\sigma'\wr(\sigma,\eta)$-flag $F=(D,\theta)$ with $|F'|+k-1$ vertices 
such that $F\downarrow_{\eta'}=F'$ where $\eta':[k']\to [k+k'-1]$ such that $\eta'(x):=x+k-1$, and the remaining edges $uu'$ of $F$ are defined as follows:
\begin{itemize}
\item if $u\in\theta([k-1])$ and $u'\in\theta([k-1])$, then $uu'$ is an edge iff $\theta^{-1}(u)\theta^{-1}(u')$ is an edge of $\sigma$,
\item if $u\in\theta([k-1])$, $u'\not\in\theta([k-1])$, then $uu'$ is an edge iff $\theta^{-1}(u)i_0$ is an edge of $\sigma$, and
\item if $u\not\in\theta([k-1])$, $u'\in\theta([k-1])$, then $uu'$ is an edge iff  $i_0\theta^{-1}(u')$ is an edge of $\sigma$.
\end{itemize}
Again, the vertices of $\theta([k-1])$ induce a digraph isomorphic to $\sigma$ restricted to $\Image(\eta)$,
the other vertices of $F$ induce the $\sigma'$-flag $F'$ (after a suitable relabeling), and
they are joined to the vertices in $\theta([k-1])$ as the vertex $i_0$ to the rest of $\sigma$.
A linear mapping $\pi^{\wr(\sigma,\eta)}_{\sigma'}:\mathbb{R}\mathcal{F}^{\sigma'} \to \mathbb{R}\mathcal{F}^{\sigma'\wr(\sigma,\eta)}$
is defined by setting $\pi^{\wr(\sigma,\eta)}_{\sigma'}(F')=F'\wr(\sigma,\eta)$ and linearly extending.
Note that $\mathcal{K}^{\sigma'}$ does not necessarily lie in the kernel of $\pi^{\wr(\sigma,\eta)}_{\sigma'}$.
As before, if $\sigma'$ is the empty type, we just write $\pi^{\wr(\sigma,\eta)}$.

Fix a type $\sigma$ of order $k$ and an injective mapping $\eta:[k-1]\to [k]$.
For a homomorphism $\Psi \in \Hom^+(\mathcal{A}^{\sigma|_{\eta}},\mathbb{R})$,
we define $\Psi^{\wr(\sigma,\eta)}:\mathcal{A}\to\mathbb{R}$ as
\begin{equation}
\Psi^{\wr(\sigma,\eta)}(F):=\left\{
\begin{array}{cl}
\Psi(\pi^{\wr(\sigma,\eta)}(F))/(\Psi((\sigma,\eta)))^{|F|} & \mbox{if $\Psi((\sigma,\eta))\not=0$, and} \\
0 & \mbox{otherwise.}
\end{array}
\right.
\label{eq-defwr}
\end{equation}
The next theorem, which follows from~\cite[Theorems 2.6 and 4.1]{FlagAlgebras}, asserts that
the just defined mapping $\Psi^{\wr(\sigma,\eta)}$ must be a homomorphism of $\mathcal{A}$ to $\mathbb{R}$,
in particular, it is well-defined.
\begin{theorem}\label{thm:piveeproperties}
Let $\sigma$ be a type of order $k$ and let $\eta:[k-1] \to [k]$ be an injective map.
\begin{enumerate}
\item[a)] For every $\Psi \in \Hom^+(\mathcal{A}^{\sigma|_{\eta}},\mathbb{R})$ we have
  $$\Psi^{\wr(\sigma,\eta)}\in \Hom^+(\mathcal{A},\mathbb{R})\;.$$
\item[b)] For every $\Psi \in \Hom^+(\mathcal{A}^{\sigma|_{\eta}},\mathbb{R})$, type
  $\sigma'$ and  $f \in \mathbb{R}\mathcal{F}^{\sigma'}_{\ell}$, $\ell\in\mathbb{N}$,  we  have
  $$\probability\left[\boldsymbol{\Psi^{\sigma'\wr(\sigma,\eta),\;\eta}}\left(\pi^{\wr(\sigma,\eta)}_{\sigma'}(f)\right) \geq 0\right]=1 \quad\Rightarrow \quad\probability\left[\boldsymbol{(\Psi^{\wr(\sigma,\eta)})}^{\boldsymbol{\sigma'}}(f) \geq 0\right]=1\;.$$
\end{enumerate}
\end{theorem}
To avoid ambiguity, we remark that
$\boldsymbol{(\Psi^{\wr(\sigma,\eta)})}^{\boldsymbol{\sigma'}}$ in Theorem~\ref{thm:piveeproperties}b) stands
for $\boldsymbol{\Phi}^{\boldsymbol{\sigma'}}$ where $\Phi=\Psi^{\wr(\sigma,\eta)}$.

\begin{remark} As an alternative to deriving Theorem~\ref{thm:piveeproperties} from~\cite[Theorems 2.6 and 4.1]{FlagAlgebras}, which are technical and are stated in the model theoretic language, one can give a direct ad hoc proof of Theorem~\ref{thm:piveeproperties}. 

For the interested reader let us, however, outline the derivation of Theorem~\ref{thm:piveeproperties} a) from the results in~\cite{FlagAlgebras}. The setting we are working in is essentially described in~\cite[Section 2.3.2]{FlagAlgebras}. We apply~\cite[Theorems 2.6]{FlagAlgebras}  to the open interpretation (as defined in~\cite[Definition 3]{FlagAlgebras}) $(U,I): T \rightsquigarrow T$, where $T$ is the theory of triangle-free digraphs. The formula $U$ represents the diagram of the flag $(\sigma,\eta) \in \mathcal{F}^{\sigma|_{\eta}}_k$ and $I$ acts identically. By~\cite[Theorems 2.6]{FlagAlgebras} applied with  $\sigma_1$ being the empty type  and $\sigma_2=\sigma|_\eta$, the map $\pi^{(U,I)}: \mathcal{A} \to  \mathcal{A}^{\sigma|_\eta}_u$, defined by linearly extending $\pi^{(U,I)}(F) = \pi^{\wr(\sigma,\eta)}(F)/(\sigma,\eta)^{|F|}$ for $F \in \mathcal{F}$, is an algebra homomorphism. We have
$\boldsymbol{\Psi^{\wr(\sigma,\eta)}} =  \boldsymbol{\Psi}\circ\pi^{(U,I)}$, implying 
Theorem~\ref{thm:piveeproperties} a).
\end{remark}

\section{Structure of triangle-free digraphs}\label{sec:Struct}
We start with several observations which will later allow
us to restrict our attention only to special classes of
digraphs.
\begin{obs}\label{obs:outregularity}
Let $D$ be a triangle-free digraph with $\delta^+(D)\ge k$. Then there exists a triangle-free digraph $D'$ on the same vertex set with outdegree of every vertex equal to $k$.
\end{obs}
Indeed, to obtain the digraph $D'$ it suffices to remove 
an arbitrary set of  $\deg^+(v)-k$ edges leaving $v$ for every vertex $v$.

We next show that a possible counterexample to Conjecture~\ref{conj:CH}
would yield counterexamples of arbitrary large order.
Suppose that $D$ is an $n$-vertex triangle-free digraph.
Replace each vertex $v\in V(D)$ by a copy $D_v$ of the digraph $D$ and
every directed edge $uv$ of the original digraph $D$ by
a complete directed bipartite graph from $D_u$ to $D_v$.
This construction yields a digraph $D'$ of order $n^2$.
Observe that $\delta^+(D')= \delta^+(D)(n+1)$ and
that if $D$ is triangle-free, then so is $D'$.
We arrive at the following observation by iterating the construction.

\begin{obs}\label{obs:scaling}
Suppose that there exists a triangle-free $n$-vertex
digraph $D$ with minimum outdegree at least $cn$. Then
for very $m_0$ there exists a triangle-free digraph $D'$ of
order $m>m_0$ with minimum outdegree at least $cm$.
Moreover, if $D$ is outregular, then so is $D'$.
\end{obs}

\subsection{Caccetta--H\"aggkvist Conjecture in the
language of flag algebras} Theorem~\ref{prop:asymptotic}
below is the main technical result of the paper. It translates
Theorem~\ref{thm:main} into the flag algebra language.
\begin{theorem}\label{prop:asymptotic}
It holds that
$$\max_{\Psi\in\mathrm{Hom}^+(\mathcal{A},\mathbb{R})}\delta_\alpha(\Psi)
< 0.3465\;.$$
\end{theorem}
Theorem~\ref{prop:asymptotic} is proven in the next section.
The maximum in Theorem~\ref{prop:asymptotic} is attained by~\cite[Theorem~3.15]{FlagAlgebras}.
We now demonstrate that it implies Theorem~\ref{thm:main}.
\begin{proof}[Proof of Theorem~\ref{thm:main}]
Suppose that there exists a
triangle-free $n$-vertex digraph $D$ with
$\delta^+(D)=cn$, $c \ge 0.3465$. By
Observation~\ref{obs:outregularity} there exists an
infinite sequence $(D_n)_{n=1}^\infty$ of
triangle-free digraphs with increasing orders such that
the digraph $D_n$ has minimum outdegree at least $c|V(D_n)|$.
By Theorem~\ref{thm:limitsexist},
there exists a subsequence $(D_{n_i})_{i=1}^\infty$ that converges
Let $\Psi\in\mathrm{Hom}^+(\mathcal{A},\mathbb{R})$ such that
$\lim_{i\rightarrow\infty} D_{n_i}=\Psi$.
By Theorem~\ref{thm:mindegreesequences}, we have $\delta_\alpha(\Psi)\ge 0.3465$ which violates Theorem~\ref{prop:asymptotic}.
\end{proof}

\subsection{Non-edges in triangle-free digraphs}
A recent result of Chudnovsky, Seymour and
Sullivan~\cite{ChuSeySull} asserts that one can delete $k$
edges from a triangle-free digraph $D$ with at most $k$
non-edges to make it acyclic. Hamburger, Haxell and
Kostochka~\cite{HHK} used this result to refine a proof of
Shen~\cite{Shen}, and consequently to obtain the previously
best-known bound on the Caccetta--H\"aggkvist conjecture.
They used the following corollary of the theorem of
Chudnovsky, Seymour and Sullivan, see~\cite{HHK} for further details.
\begin{lemma}\label{lem:HHKdegree}
Suppose $D$ is a triangle-free digraph with $k$
non-edges. Then there is a vertex $v\in V(D)$ with
$\deg^+(v)<\sqrt{2k}$.
\end{lemma}
Lemma~\ref{lem:HHKdegree} can be stated in the language of flag algebras as follows.
\begin{lemma}\label{lem:HHKdegree2}
For every $\Psi \in \Hom^{+}(\mathcal{A},\mathbb{R})$
 we have for every $\epsilon_0>0$
that
$$\probability\left[\boldsymbol{\Psi^{\lambda}}(\alpha)<\sqrt{1-\Psi(\varrho)}+\epsilon_0\right]>0\;.$$
\end{lemma}
\begin{proof}
Let $\{D_n\}_{n=1}^\infty$ be a sequence of triangle-free digraphs such that
$\lim_{n\to\infty} D_n=\Psi$, which exists by Theorem~\ref{thm:limitsexist}.
The definition of the convergence of a sequence of digraphs yields that
for every $\epsilon>0$, there exists a number $n_0$ such that
every digraph $D_n$ ($n>n_0$) contains at most $(1-\Psi(\varrho)+\epsilon)|V(D_n)|^2/2$ non-edges.
A repeated application of Lemma~\ref{lem:HHKdegree} yields the existence of a set $S_n\subseteq V(D_n)$, $|S_n|\ge \epsilon |V(D_n)|$,
such that $\deg^+(v)\le \left(1-\varepsilon\right)^{-1}\sqrt{1-\Psi(\varrho)+\epsilon}|V(D_n)|$ for every $v\in S_n$.
We conclude using Lemma~\ref{lem:limitoutdegDistr} that $$\probability\left[\boldsymbol{\Psi^{\lambda}}(\alpha)\le
\left(1-\epsilon\right)^{-1}\sqrt{1-\Psi(\varrho)+\epsilon}\right]\ge
\epsilon\;\mbox{.}$$
The statement of the lemma now follows by choosing $\epsilon$ sufficiently small.
\end{proof}

\section{Proof of Theorem~\ref{prop:asymptotic}}\label{sec:ProofofProp}
Fix $\Psi\in \mathrm{Hom}^+(\mathcal{A},\mathbb{R})$
with the maximum value of $\delta_\alpha(\Psi)$ in $\mathrm{Hom}^+(\mathcal{A},\mathbb{R})$
for the rest of the section
(recall that the maximum is attained by~\cite[Theorem~3.15]{FlagAlgebras} as we have already mentioned).
Set $c_0:=\delta_\alpha(\Psi)$. By considering a sequence of
digraphs converging to $\Psi$ and using
Theorems~\ref{thm:limitsexist}
and~\ref{thm:mindegreeconvergence}, and
Observation~\ref{obs:outregularity}, we can assume without loss of generality that
\begin{equation}\label{eq:outdegc}
\probability[\boldsymbol{\Psi^\lambda}(\alpha)=c_0]=1\;.
\end{equation}
We proceed by deriving a series of inequalities on the
values of $\Psi$ in Subsections~\ref{sec:CauchySchwarz}--\ref{sec:Fork}.

We will
concentrate our attention on inequalities which can be
expressed in terms of values of $\Psi$ on the elements of
$\mathcal{F}_{4}$. To be able to write  down these
inequalities we will need to enumerate elements of
$\calF_4$, and also the elements of $\calF^\beta_3$ and
$\calF^\lambda_3$, which is done in the following subsection.

\subsection{Notation}
We now fix notation for the elements of $\calF^\beta_3$, $\calF^\lambda_3$ and $\calF_4$.
The elements of  $\calF^\beta_3$ are denoted by $K_0,\ldots, K_7$.
The vertex set of each of these digraphs is $\{1,2,a\}$, where $1$ and $2$ are the vertices of $\beta$,
and the edge set of each of these digraphs is listed in the table below.
We will write $\overline{K}$ for $( K_0, \ldots, K_7 )\in\left(\mathcal A^{\beta}\right)^8$.

\bigskip
\begin{tabular*}{5cm}{ll}
    $K_0$& $\{12\}$  \\
    $K_1$& $\{12,2a\}$  \\
    $K_2$& $\{12,a2\}$  \\
\end{tabular*}
\begin{tabular*}{5cm}{ll}
    $K_3$& $\{12,1a\}$  \\
    $K_4$& $\{12,1a,2a\}$  \\
    $K_5$& $\{12,1a,a2\}$  \\
\end{tabular*}
\begin{tabular*}{5cm}{ll}
    $K_6$& $\{12,a1\}$  \\
    $K_7$& $\{12,a1,a2\}$  \\
\end{tabular*}

\bigskip

The symbols $L_0,\ldots L_{13}$  denote the elements of $\calF^\lambda_3$
which are considered as digraphs with the vertex set $\{1,a,b\}$, where $1$ is corresponds to the vertex of $\lambda$.
The edge sets are as follows.
\bigskip

\begin{tabular*}{5cm}{ll}
    $L_0$& $\{						\}$  \\
    $L_1$& $\{ ab					\}$  \\
    $L_2$& $\{ 1b					\}$  \\
    $L_3$& $\{ 1b,ab					\}$  \\
    $L_4$& $\{ 1b,ba					\}$  \\
\end{tabular*}
\begin{tabular*}{5cm}{ll}
    $L_5$& $\{ b1			\}$  \\
    $L_6$& $\{ b1,ab			\}$  \\
    $L_7$& $\{ b1,ba			\}$  \\
    $L_8$& $\{ 1a,1b				\}$  \\
    $L_9$& $\{ 1a,1b,ab			\}$  \\
\end{tabular*}
\begin{tabular*}{5cm}{ll}
    $L_{10}$& $\{ 1a,b1				\}$  \\
    $L_{11}$& $\{ 1a,b1,ba				\}$  \\
    $L_{12}$& $\{ a1,b1				\}$  \\
    $L_{13}$& $\{ a1,b1,ab			\}$  \\
\end{tabular*}

\bigskip
\noindent As in the previous case, we will use $\overline{L}$ for $( L_0, \ldots, L_{13})\in(\mathcal{A}^{\lambda})^{14}$.

Finally, we enumerate the elements of $\calF_4$, i.e.
all isomorphism types of triangle-free digraphs on the
vertex set $\{a,b,c,d\}$. The table below gives edge sets
of each of these thirty-two digraphs $H_0,\ldots,H_{31}$.

\bigskip

\begin{tabular*}{5cm}{ll}
    $H_0$& $\{						\}$  \\
    $H_1$& $\{ cd					\}$  \\
    $H_2$& $\{ bd,cd					\}$  \\
    $H_3$& $\{ bd,dc					\}$  \\
    $H_4$& $\{ db,dc					\}$  \\
    $H_5$& $\{ ad,bd,cd				\}$  \\
    $H_6$& $\{ ad,bd,dc				\}$  \\
    $H_7$& $\{ ad,db,dc				\}$  \\
    $H_8$& $\{ da,db,dc				\}$  \\
    $H_9$& $\{ bc,bd,cd				\}$  \\
    $H_{10}$& $\{ ad,bc				\}$  \\
\end{tabular*}
\begin{tabular*}{5cm}{ll}
    $H_{11}$& $\{ ad,bc,cd				\}$  \\
    $H_{12}$& $\{ ad,bc,bd				\}$  \\
    $H_{13}$& $\{ ad,bc,bd,cd			\}$  \\
    $H_{14}$& $\{ ad,bc,bd,dc			\}$  \\
    $H_{15}$& $\{ ad,bc,db				\}$  \\
    $H_{16}$& $\{ ad,bc,db,dc			\}$  \\
    $H_{17}$& $\{ da,bc,bd				\}$  \\
    $H_{18}$& $\{ da,bc,bd,cd			\}$  \\
    $H_{19}$& $\{ da,bc,bd,dc			\}$  \\
    $H_{20}$& $\{ da,bc,db,dc			\}$  \\
    $H_{21}$& $\{ ac,ad,bc,bd			\}$  \\
\end{tabular*}
\begin{tabular*}{5cm}{ll}
    $H_{22}$& $\{ ac,ad,bc,bd,cd		\}$  \\
    $H_{23}$& $\{ ac,ad,bc,db			\}$  \\
    $H_{24}$& $\{ ac,ad,bc,db,dc		\}$  \\
    $H_{25}$& $\{ ac,da,bc,db			\}$  \\
    $H_{26}$& $\{ ac,da,bc,db,dc		\}$  \\
    $H_{27}$& $\{ ac,ad,cb,db,cd		\}$  \\
    $H_{28}$& $\{ ac,da,cb,bd			\}$  \\
    $H_{29}$& $\{ ac,da,cb,db,dc		\}$  \\
    $H_{30}$& $\{ ca,da,cb,db,cd		\}$  \\
    $H_{31}$& $\{ ab,ac,ad,bc,bd,cd     \}$
\end{tabular*}

\bigskip

\noindent If $\Psi\in\Hom^+(\mathcal A,\mathbb R)$,
we will write $\Psi_i$ instead of $\Psi(H_i)$ for brevity.
So, we can view $\overline{\Psi}=(\Psi_i)_{i=0}^{31}$ as an element of $\mathbb{R}^{32}$.

\subsection{Cauchy--Schwarz inequalities}\label{sec:CauchySchwarz}
Let $\overline{a} \in \mathbb{R}^8$ be a (row) vector.
A direct computation gives that $24\Psi(\lsem (\overline{a}\overline{K}^{\top})^2 \rsem_{\beta})
= \overline{a}(A_C(\overline{\Psi}))\overline{a}^{\top}$
where $A_C(\overline{\Psi})$ is the matrix given in Table~\ref{tab:CS};
the entry $A_C(\overline{\Psi})_{ij}$ is $\Psi(24\lsem K_i \cdot K_j \rsem_{\beta})$
computed by expressing $\lsem K_i \cdot K_j \rsem_{\beta} \in \mathcal A$  as a sum of elements of $\calF_4$.

\begin{table}
{\fontsize{5pt}{0.10pt}\selectfont
$$
\arraycolsep=0.3pt
\left[
\begin{array}{cccccccc}
2\Psi_1 + 4\Psi_{10} & \Psi_3 + \Psi_{11} + \Psi_{15} & 2\Psi_2 + \Psi_{11} + \Psi_{12}   & 2\Psi_4 + \Psi_{12} + \Psi_{17}
& \Psi_9 + \Psi_{13}+ \Psi_{18} & \Psi_9 + \Psi_{14}+ \Psi_{19} & \Psi_3+\Psi_{15}+ \Psi_{17} & \Psi_9 +\Psi_{16}+\Psi_{20}\\[6pt]
\Psi_3 + \Psi_{11} + \Psi_{15} & 2\Psi_7+2\Psi_{16} & 2\Psi_6+\Psi_{14} & \Psi_{17}+\Psi_{23}+2\Psi_{25}
& \Psi_{19}+\Psi_{24}+\Psi_{27} & \Psi_{18}+\Psi_{27} & \Psi_{15}+\Psi_{23}+4\Psi_{28} & \Psi_{18}+\Psi_{29} \\[6pt]
2\Psi_2 + \Psi_{11} + \Psi_{12}  & 2\Psi_6+\Psi_{14} & 6\Psi_5 +2\Psi_{13} & \Psi_{12}+4\Psi_{21}+\Psi_{23}
& \Psi_{14}+2\Psi_{22} & \Psi_{13}+2\Psi_{22}+\Psi_{24} & \Psi_{11}+\Psi_{23}+2\Psi_{25} & \Psi_{13}+\Psi_{24} +2\Psi_{26}\\[6pt]
2\Psi_4 + \Psi_{12} + \Psi_{17} & \Psi_{17}+\Psi_{23}+2\Psi_{25} & \Psi_{12}+4\Psi_{21}+\Psi_{23}  & 6\Psi_8+2\Psi_{20}
& \Psi_{20}+2\Psi_{26}+\Psi_{29} & \Psi_{20}+\Psi_{29}+2\Psi_{30} & 2\Psi_7+\Psi_{19} & \Psi_{19}+2\Psi_{30} \\[6pt]
\Psi_9 + \Psi_{13}+ \Psi_{18} & \Psi_{19}+\Psi_{24}+\Psi_{27} &  \Psi_{14}+2\Psi_{22} & \Psi_{20}+2\Psi_{26}+\Psi_{29}
& 2\Psi_{30}+2\Psi_{31} & \Psi_{29}+\Psi_{31} & \Psi_{16}+\Psi_{24} & \Psi_{27}+\Psi_{31} \\[6pt]
\Psi_9 + \Psi_{14}+ \Psi_{19} & \Psi_{18}+\Psi_{27} & \Psi_{13}+2\Psi_{22}+\Psi_{24} & \Psi_{20}+\Psi_{29}+2\Psi_{30}
& \Psi_{29}+\Psi_{31} & 2\Psi_{26}+2\Psi_{31} & \Psi_{16}+\Psi_{27} & \Psi_{24}+\Psi_{31} \\[6pt]
\Psi_3+\Psi_{15}+ \Psi_{17} & \Psi_{15}+\Psi_{23}+4\Psi_{28}  & \Psi_{11}+\Psi_{23}+2\Psi_{25} & 2\Psi_7+\Psi_{19}
& \Psi_{16}+\Psi_{24} & \Psi_{16}+\Psi_{27} & 2\Psi_6 + 2\Psi_{18} & \Psi_{14}+\Psi_{27}+\Psi_{29}\\[6pt]
\Psi_9 +\Psi_{16}+\Psi_{20} & \Psi_{18}+\Psi_{29} & \Psi_{13}+\Psi_{24} +2\Psi_{26} & \Psi_{19}+2\Psi_{30}
& \Psi_{27}+\Psi_{31} & \Psi_{24}+\Psi_{31} & \Psi_{14}+\Psi_{27}+\Psi_{29} & 2\Psi_{22}+2\Psi_{31} \\
\end{array}
\right]$$
}
\caption{The matrix $A_C(\overline{\Psi})$.}
\label{tab:CS}
\end{table}

From Theorem~\ref{thm:CauchySchwarz} we deduce the following.

\begin{cor}\label{cor:CauchySchwarz}
We have
\begin{equation}\label{eq:CS}
\overline{a}(A_C(\overline{\Psi}))\overline{a}^{\top} \geq 0
\end{equation}
for every $\Psi\in\Hom^+(\mathcal A,\mathbb R)$ and every $\overline{a} \in \mathbb{R}^8$.
\end{cor}

\subsection{Outregularity}\label{sec:Regularity}
In this section, we give a simple corollary of~\eqref{eq:outdegc} to frequencies of more general subgraphs.
\begin{lemma}\label{lem:outreg}
We have $\Psi(\lsem f \cdot (\alpha - c_0)\rsem_\lambda) = 0$ for every $f \in \mathcal{A}^\lambda$.
\end{lemma}
\begin{proof}
Since $\Psione(\alpha-c_0)=0$ with probability one by \eqref{eq:outdegc} and $\Psione\in\Hom^+(\mathcal{A}^\lambda,\mathbb R)$,
it follows that $\Psione( f \cdot (\alpha - c_0))=0$ with probability one for every $f \in \mathcal{A}^\lambda$.
The statement now follows from~(\ref{eq:Razb21}).
\end{proof}
For $\overline{b} \in \mathbb{R}^{14}$,
we get that
$$24\Psi(\lsem \overline{b}\overline{L}^{\top} \cdot (\alpha - c_0)\rsem_\lambda) = \overline{b}(B_{R}-c_0A_{R})\overline{\Psi}^{\top}$$
where the matrices $A_{R}$ and $B_{R}$ are given in Table~\ref{tab:reg}.
Lemma~\ref{lem:outreg} thus implies the following.

\begin{table}
{\footnotesize
$$A_{R}=\left[
\arraycolsep=2.5pt
\begin{array}{cccccccccccccccccccccccccccccccc}
12 & 6 & 3 & 3 & 3 & 3 & 3 & 3 & 3 & 0 & 0 & 0 & 0 & 0 & 0 & 0 & 0 & 0 & 0 & 0 & 0 & 0 & 0 & 0 & 0 & 0 & 0 & 0 & 0 & 0 & 0 & 0 \\
 0 & 2 & 2 & 2 & 2 & 0 & 0 & 0 & 0 & 3 & 4 & 2 & 2 & 1 & 1 & 2 & 1 & 2 & 1 & 1 & 1 & 0 & 0 & 0 & 0 & 0 & 0 & 0 & 0 & 0 & 0 & 0 \\
 0 & 2 & 2 & 2 & 2 & 0 & 0 & 0 & 0 & 3 & 4 & 2 & 2 & 1 & 1 & 2 & 1 & 2 & 1 & 1 & 1 & 0 & 0 & 0 & 0 & 0 & 0 & 0 & 0 & 0 & 0 & 0 \\
 0 & 0 & 2 & 0 & 0 & 6 & 2 & 0 & 0 & 0 & 0 & 2 & 2 & 4 & 2 & 0 & 0 & 0 & 0 & 0 & 0 & 4 & 4 & 2 & 2 & 2 & 2 & 0 & 0 & 0 & 0 & 0 \\
 0 & 0 & 0 & 1 & 0 & 0 & 2 & 2 & 0 & 0 & 0 & 1 & 0 & 0 & 1 & 2 & 2 & 1 & 2 & 1 & 0 & 0 & 0 & 2 & 1 & 2 & 0 & 2 & 4 & 1 & 0 & 0 \\
 0 & 2 & 2 & 2 & 2 & 0 & 0 & 0 & 0 & 3 & 4 & 2 & 2 & 1 & 1 & 2 & 1 & 2 & 1 & 1 & 1 & 0 & 0 & 0 & 0 & 0 & 0 & 0 & 0 & 0 & 0 & 0 \\
 0 & 0 & 0 & 1 & 0 & 0 & 2 & 2 & 0 & 0 & 0 & 1 & 0 & 0 & 1 & 2 & 2 & 1 & 2 & 1 & 0 & 0 & 0 & 2 & 1 & 2 & 0 & 2 & 4 & 1 & 0 & 0 \\
 0 & 0 & 0 & 0 & 2 & 0 & 0 & 2 & 6 & 0 & 0 & 0 & 2 & 0 & 0 & 0 & 0 & 2 & 0 & 2 & 4 & 4 & 0 & 2 & 0 & 2 & 2 & 0 & 0 & 2 & 4 & 0 \\
 0 & 0 & 0 & 0 & 1 & 0 & 0 & 1 & 3 & 0 & 0 & 0 & 1 & 0 & 0 & 0 & 0 & 1 & 0 & 1 & 2 & 2 & 0 & 1 & 0 & 1 & 1 & 0 & 0 & 1 & 2 & 0 \\
 0 & 0 & 0 & 0 & 0 & 0 & 0 & 0 & 0 & 1 & 0 & 0 & 0 & 1 & 1 & 0 & 1 & 0 & 1 & 1 & 1 & 0 & 2 & 0 & 2 & 0 & 2 & 2 & 0 & 2 & 2 & 4 \\
 0 & 0 & 0 & 1 & 0 & 0 & 2 & 2 & 0 & 0 & 0 & 1 & 0 & 0 & 1 & 2 & 2 & 1 & 2 & 1 & 0 & 0 & 0 & 2 & 1 & 2 & 0 & 2 & 4 & 1 & 0 & 0 \\
 0 & 0 & 0 & 0 & 0 & 0 & 0 & 0 & 0 & 1 & 0 & 0 & 0 & 1 & 1 & 0 & 1 & 0 & 1 & 1 & 1 & 0 & 2 & 0 & 2 & 0 & 2 & 2 & 0 & 2 & 2 & 4 \\
 0 & 0 & 1 & 0 & 0 & 3 & 1 & 0 & 0 & 0 & 0 & 1 & 1 & 2 & 1 & 0 & 0 & 0 & 0 & 0 & 0 & 2 & 2 & 1 & 1 & 1 & 1 & 0 & 0 & 0 & 0 & 0 \\
 0 & 0 & 0 & 0 & 0 & 0 & 0 & 0 & 0 & 1 & 0 & 0 & 0 & 1 & 1 & 0 & 1 & 0 & 1 & 1 & 1 & 0 & 2 & 0 & 2 & 0 & 2 & 2 & 0 & 2 & 2 & 4
\end{array}
\right]
$$

$$B_{R}=\left[
\arraycolsep=2.5pt
\begin{array}{cccccccccccccccccccccccccccccccc}
 0 & 1 & 2 & 1 & 0 & 3 & 2 & 1 & 0 & 0 & 0 & 0 & 0 & 0 & 0 & 0 & 0 & 0 & 0 & 0 & 0 & 0 & 0 & 0 & 0 & 0 & 0 & 0 & 0 & 0 & 0 & 0 \\
 0 & 0 & 0 & 0 & 0 & 0 & 0 & 0 & 0 & 0 & 2 & 2 & 1 & 1 & 1 & 1 & 1 & 0 & 0 & 0 & 0 & 0 & 0 & 0 & 0 & 0 & 0 & 0 & 0 & 0 & 0 & 0 \\
 0 & 0 & 0 & 0 & 2 & 0 & 0 & 0 & 0 & 2 & 0 & 0 & 1 & 1 & 1 & 0 & 0 & 1 & 1 & 1 & 0 & 0 & 0 & 0 & 0 & 0 & 0 & 0 & 0 & 0 & 0 & 0 \\
 0 & 0 & 0 & 0 & 0 & 0 & 0 & 0 & 0 & 0 & 0 & 0 & 1 & 1 & 1 & 0 & 0 & 0 & 0 & 0 & 0 & 4 & 4 & 1 & 1 & 0 & 0 & 0 & 0 & 0 & 0 & 0 \\
 0 & 0 & 0 & 0 & 0 & 0 & 0 & 0 & 0 & 0 & 0 & 0 & 0 & 0 & 0 & 0 & 0 & 1 & 1 & 1 & 0 & 0 & 0 & 1 & 1 & 2 & 0 & 2 & 0 & 0 & 0 & 0 \\
 0 & 0 & 0 & 1 & 0 & 0 & 0 & 0 & 0 & 1 & 0 & 1 & 0 & 1 & 0 & 1 & 0 & 0 & 1 & 0 & 0 & 0 & 0 & 0 & 0 & 0 & 0 & 0 & 0 & 0 & 0 & 0 \\
 0 & 0 & 0 & 0 & 0 & 0 & 0 & 0 & 0 & 0 & 0 & 0 & 0 & 0 & 0 & 1 & 1 & 0 & 0 & 0 & 0 & 0 & 0 & 1 & 1 & 0 & 0 & 0 & 4 & 0 & 0 & 0 \\
 0 & 0 & 0 & 0 & 0 & 0 & 0 & 0 & 0 & 0 & 0 & 0 & 0 & 0 & 0 & 0 & 0 & 1 & 0 & 0 & 1 & 0 & 0 & 1 & 0 & 2 & 2 & 0 & 0 & 1 & 0 & 0 \\
 0 & 0 & 0 & 0 & 0 & 0 & 0 & 0 & 3 & 0 & 0 & 0 & 0 & 0 & 0 & 0 & 0 & 0 & 0 & 0 & 2 & 0 & 0 & 0 & 0 & 0 & 1 & 0 & 0 & 1 & 1 & 0 \\
 0 & 0 & 0 & 0 & 0 & 0 & 0 & 0 & 0 & 0 & 0 & 0 & 0 & 0 & 0 & 0 & 0 & 0 & 0 & 0 & 1 & 0 & 0 & 0 & 0 & 0 & 2 & 0 & 0 & 2 & 2 & 3 \\
 0 & 0 & 0 & 0 & 0 & 0 & 0 & 2 & 0 & 0 & 0 & 0 & 0 & 0 & 0 & 0 & 2 & 0 & 0 & 1 & 0 & 0 & 0 & 0 & 1 & 0 & 0 & 1 & 0 & 0 & 0 & 0 \\
 0 & 0 & 0 & 0 & 0 & 0 & 0 & 0 & 0 & 0 & 0 & 0 & 0 & 0 & 0 & 0 & 0 & 0 & 0 & 1 & 0 & 0 & 0 & 0 & 1 & 0 & 0 & 1 & 0 & 0 & 2 & 2 \\
 0 & 0 & 0 & 0 & 0 & 0 & 1 & 0 & 0 & 0 & 0 & 0 & 0 & 0 & 1 & 0 & 0 & 0 & 0 & 0 & 0 & 0 & 1 & 0 & 0 & 0 & 0 & 0 & 0 & 0 & 0 & 0 \\
 0 & 0 & 0 & 0 & 0 & 0 & 0 & 0 & 0 & 0 & 0 & 0 & 0 & 0 & 0 & 0 & 0 & 0 & 1 & 0 & 0 & 0 & 0 & 0 & 0 & 0 & 0 & 1 & 0 & 1 & 0 & 1
\end{array}
\right]
$$
}
\caption{The matrices $A_{R}$ and $B_{R}$.}
\label{tab:reg}
\end{table}

\begin{cor}\label{cor:outreg}
We have $\overline{b}(B_{R}-c_0A_{R})\overline{\Psi}^{\top} \geq 0$
for every $\overline{b} \in \mathbb{R}^{14}$.
\end{cor}

Note that the inequality in Corollary~\ref{cor:outreg} always holds with an equality.

\subsection{Induction}\label{sec:Induction}
In this section we formalize the inductive argument of
Shen~\cite{Shen} in the language of flag algebras and
generalize it. Let $F=(D,\theta)$ be a  $\sigma$-flag. We
say that $F$ is a $\sigma$-\emph{source}
if $D$ contains no edge from $V(D)\setminus\Image(\theta)$ to $\Image(\theta)$.
The set of all $\sigma$-sources of order $k$ is denoted by $\mathcal{F}^{\sigma,\to}_{k}$.

Recall that $c_0=\delta_\alpha(\Psi)$ where $\Psi$ is the homomorphism fixed at the beginning of the section.
Let $\sigma$  be a type of order $k$ such that the vertex $1$ has indegree $k-1$ in $\sigma$, and
let $F_0=(D,\theta)$ be the $\sigma$-flag with $|V(D)|=k+1$ such that
every vertex of $\Image(\theta)$ is connected to the unique vertex in $V(D)\setminus\Image(\theta)$ by an outgoing edge.
Set
\begin{equation}\label{eq:indass} f(\sigma) := -c_0 +
\sum\{F : \; F \in \mathcal{F}^{\sigma,\to}_{k+1}, \; F
\not \cong F_0\} + c_0F_0\in{\mathcal A}^{\sigma}\;.
\end{equation}

\begin{lemma}\label{lem:induction}
We have $\Psi(\lsem f(\sigma) \rsem_{\sigma}) \geq 0$ for every type $\sigma$ where the vertex $1$ has indegree $|\sigma|-1$.
\end{lemma}
\begin{proof}
We distinguish two cases. If $\Psi((\sigma,0))=0$, then $\Psi(\lsem f(\sigma) \rsem_{\sigma})=0$.
Therefore, it suffices to consider the case $\Psi((\sigma,0))>0$.
We infer from \eqref{eq:Razb21} that the assertion of the lemma would follow if we proved that
\begin{equation*}
\probability[\boldsymbol{\Psi^{\sigma}}(f(\sigma))
\geq 0]=1\;.
\end{equation*}
We now prove this equality.

Let $k$ be the order of $\sigma$,
let $F_0=(D,\theta)$ be the $\sigma$-flag with $|V(D)|=k+1$ such that
every vertex of $\Image(\theta)$ is connected to the unique vertex in $V(D)\setminus\Image(\theta)$ by an outgoing edge, and
let $\send_i:[1]\rightarrow [k+1]$ be the mapping such that $\eta_i(1):=i$.
If $\Phi \in \Hom^{+}(\mathcal{A}^{\sigma},\mathbb{R})$ is such that
$\Phi(F_0)=0$, then we obtain using triangle-freeness (see Figure~\ref{fig:trianglefree}) that
\begin{equation}\label{eq:almf}\Phi(\pi^{\sigma,\send_1}(\alpha))
\leq \Phi\left(\sum\{F : \; F \in
\mathcal{F}^{\sigma,\to}_{k+1}, \; F \not \cong
F_0\}\right)\;.
\end{equation}
\begin{figure}
  \centering
\includegraphics{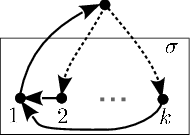}
\caption{The algebra element $\pi^{\sigma,\send_1}(\alpha)$ in~\eqref{eq:almf} expands into a sum of $2^{k-1}$ flags, each of which is in $\mathcal{F}^{\sigma,\to}_{k+1}$.}
  \label{fig:trianglefree}
\end{figure}
Theorem~\ref{thm:piproperties}~b) and (\ref{eq:outdegc}) (recall that $\Psi((\sigma,0))>0$) yield that
 $\boldsymbol{\Psi^{\sigma}}(\pi^{\sigma,\send_1}(\alpha))=c_0$ holds with probability one.
This combines with \eqref{eq:indass} and~\eqref{eq:almf} to the following:
$$\probability[\boldsymbol{\Psi^{\sigma}}(f(\sigma)) < 0 \:\&\: \boldsymbol{\Psi^{\sigma}}(F_0)=0]=0\;\mbox{.}$$
Therefore, it suffices to show that
\begin{equation}\label{eq:ind1}
\probability[\boldsymbol{\Psi^{\sigma}}(f(\sigma)) < 0
\:\&\: \boldsymbol{\Psi^{\sigma}}(F_0)>0]=0\;.
\end{equation}
We restrict our attention to the case $\probability[\boldsymbol{\Psi^{\sigma}}(F_0)>0]>0$;
otherwise, \eqref{eq:ind1} is zero.
Let $\sigma'$ be the (unique) type of order $k+1$
such that $(\sigma',\eta') \cong_{\sigma} F_0$
where $\eta'$ is the identity mapping from $[k]$ to itself.
Theorem~\ref{thm:piproperties}~b) and (\ref{eq:outdegc}) imply that
\begin{equation}
\probability[\boldsymbol{\Psi^{\sigma'}}(\pi^{\sigma',\send_{k+1}}(\alpha-c_0))=0]=1\;.\label{eq-ts1}
\end{equation}
It follows that
\begin{equation}
\probability^{\sigma}\big[\probability^{\sigma',\eta'}[\boldsymbol{(\Psi^{\sigma})^{\sigma',\eta'}}(\pi^{\sigma',\send_{k+1}}(\alpha-c_0))=0]=1\big]=1\;.\label{eq-ts2}
\end{equation}
Indeed, the transition from (\ref{eq-ts1}) to (\ref{eq-ts2}) corresponds to sampling $\sigma'$ in two steps by sampling first $\sigma$ and then extending the sample to $\sigma'$.

Thus \eqref{eq:ind1} will be established by the following claim.
\begin{claim}\label{cl:sp}
For every  $\Phi \in \Hom^{+}(\mathcal{A}^{\sigma},\mathbb{R})$ with $\Phi(F_0)>0$ and
\begin{equation}\label{eq:assCl}
\probability[\boldsymbol{\Phi^{\sigma',\eta'}}(\pi^{\sigma',\send_{k+1}}(\alpha-c_0))=0]=1
\end{equation}
we have
\begin{equation}\label{qmf}
\Phi(f(\sigma)) \geq 0\;.
\end{equation}
\end{claim}
\begin{proof}[Proof of Claim~\ref{cl:sp}]
Let $\omega=\alpha+\overline{\alpha}+\gamma$  be the formal sum of the three elements of $\calF^\lambda_2$.
Note that $\omega=\lambda$ if considered as an element of $\mathcal{A}^\lambda$
but we are going to apply the mapping $\pi^{\wr(\sigma',\eta')}_{\lambda}$ and
$\pi^{\wr(\sigma',\eta')}_{\lambda}(\lambda)=\sigma'\not=\pi^{\wr(\sigma',\eta')}_{\lambda}(\omega)$.
Also note that
$\pi^{\sigma',\eta'}(F_0)=\pi^{\wr(\sigma',\eta')}_{\lambda}(\omega)$ (see Figure~\ref{fig:picturea} for illustration).
\begin{figure}
  \centering
\includegraphics[scale=0.92]{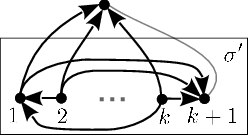}
\caption{The algebra element $\pi^{\sigma',\eta'}(F_0)=\pi^{\wr(\sigma',\eta')}_{\lambda}(\omega)$.}
  \label{fig:picturea}
\end{figure}
For every $\epsilon>0$, we have
\begin{align}\label{eq:ind2}
\pi^{\sigma',\eta'}(f(\sigma)+\epsilon) &=
\pi^{\sigma',\eta'}\left(\sum\{F : \; F \in
\mathcal{F}^{\sigma,\to}_{k+1}, \; F \not \cong F_0\} +
c_0F_0-c_0+\epsilon\right) \notag\\ &\geq_{\sigma'}
\pi^{\sigma',\send_{k+1}}(\alpha-c_0) -
\pi^{\wr(\sigma',\eta')}_{\lambda}(\alpha-(c_0+\epsilon)\omega)\;,
\end{align}
since the digraphs considered are triangle-free (see Figure~\ref{fig:pictureb}).
\begin{figure}
  \centering
\includegraphics[scale=0.92]{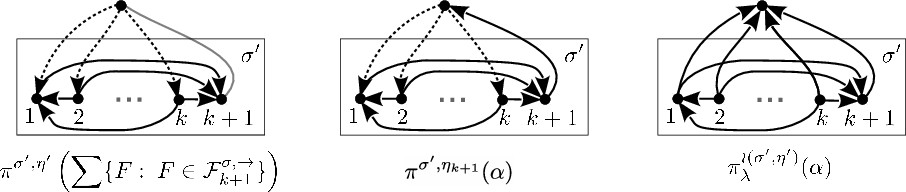}
\caption{Algebra elements appearing in~\eqref{eq:ind2}.}
  \label{fig:pictureb}
\end{figure}
 It follows that 
\begin{equation}\label{eq:ind3}
\probability\left[\boldsymbol{\Phi^{\sigma',\eta'}}\left(\pi^{\sigma',\eta'}(f(\sigma)+\epsilon)\right)
\geq
\boldsymbol{\Phi^{\sigma',\eta'}}\left(\pi^{\sigma',\send_{k+1}}(\alpha-c_0)
-
\pi^{\wr(\sigma',\eta')}_{\lambda}(\alpha-(c_0+\epsilon)\omega)\right)\right]=1\;.
\end{equation}
The identity (\ref{eq:ind3}) can be rewritten using~\eqref{eq:assCl} as
\begin{equation}\label{eq:PrPo}
\probability\left[\boldsymbol{\Phi^{\sigma',\eta'}}\left(\pi^{\sigma',\eta'}(f(\sigma)+\epsilon)\right)
\geq
\boldsymbol{\Phi^{\sigma',\eta'}}\left(\pi^{\wr(\sigma',\eta')}_{\lambda}((c_0+\epsilon)\omega-\alpha)\right)\right]=1\;.
\end{equation}
We claim that
\begin{equation}\label{eq:kos}
\probability[\boldsymbol{\Phi^{\sigma',\eta'}}(\pi^{\wr(\sigma',\eta')}_{\lambda}((c_0+\epsilon)\omega-\alpha))\ge 0]>0\;.
\end{equation}
Suppose that~\eqref{eq:kos} fails for some $\epsilon>0$. Then $\probability[\boldsymbol{\Phi^{\sigma',\eta'}}(\pi^{\wr(\sigma',\eta')}_{\lambda}(\alpha-(c_0+\epsilon)\omega)) > 0]=1$, and Theorem~\ref{thm:piveeproperties}~b) would imply that $\probability[\boldsymbol{(\Phi^{\wr(\sigma',\eta')})^\lambda}(\alpha) \ge c_0+\epsilon]=1$.
This would give that $\delta_{\alpha}(\Phi^{\wr(\sigma',\eta')}) \geq c_0+\epsilon$,
which would contradict the choice of~$\Psi$ as
a homomorphism with the maximum possible value of $\delta_\alpha(\Psi)$ in $\mathrm{Hom}^+(\mathcal{A},\mathbb{R})$.
So, \eqref{eq:kos} indeed holds.

Combining~\eqref{eq:PrPo} and~\eqref{eq:kos} gives that
$$
\probability[\boldsymbol{\Phi^{\sigma',\eta'}}\left(\pi^{\sigma',\eta'}(f(\sigma)+\epsilon)\right) \geq 0]=
\probability[\boldsymbol{\Phi^{\sigma',\eta'}}\left(\pi^{\sigma',\eta'}(f(\sigma))\right)+\epsilon \geq 0]
>0\;\mbox{.}$$
The claimed inequality \eqref{qmf} follows by considering $\epsilon \searrow 0$ and using Theorem~\ref{thm:piproperties}~a),
which finishes the proof of Claim~\ref{cl:sp}.
\end{proof}
The proof of Lemma~\ref{lem:induction} is now also finished.
\end{proof}

Let $T$ and $V$ be the types of order 3 with $E(T)=\{23,21,31\}$ and $E(V)=\{21,31\}$, respectively.
Note that $T$ and $V$ are the only types of order $3$ satisfying the conditions of Lemma~\ref{lem:induction}.  We have
\begin{align}
\begin{split}\label{eq:IndT}
24\cdot\Psi(\lsem f(T) \rsem_T) = &(1-c_0)\Psi_9-c_0\Psi_{13}-c_0\Psi_{14}-c_0\Psi_{16}+(1-c_0)\Psi_{18}+(1-c_0)\Psi_{19}+(1-c_0)\Psi_{20}\\
  & -2c_0\Psi_{22}-2c_0\Psi_{24}-2c_0\Psi_{26}+(1-2c_0)\Psi_{27} +(1-2c_0)\Psi_{29}\\
  & +(2-2c_0)\Psi_{30}-3c_0\Psi_{31}\;,\mbox{and}
\end{split}\\
\begin{split}\label{eq:IndV}
12\cdot\Psi(\lsem f(V) \rsem_V) = &(1-c_0)\Psi_2-3c_0\Psi_5+(1-c_0)\Psi_6-c_0\Psi_{11}+(1-c_0)\Psi_{12}-2c_0\Psi_{13}+(1-c_0)\Psi_{14}\\
&+(2-2c_0)\Psi_{21}-c_0\Psi_{22}-c_0\Psi_{23}-c_0\Psi_{24}-c_0\Psi_{25}-c_0\Psi_{26}\;.
\end{split}
\end{align}

The right hand sides of the identities (\ref{eq:IndT}) and (\ref{eq:IndV}), which depend on $\overline{\Psi}$,
are denoted by $\Ind_T(\overline{\Psi})$ and $\Ind_V(\overline{\Psi})$, respectively.
Lemma~\ref{lem:induction} gives the following.

\begin{cor}\label{cor:induction} $\Ind_T(\overline{\Psi}) \geq 0$ and $\Ind_V(\overline{\Psi}) \geq 0$.
\end{cor}

\subsection{Density of forks}\label{sec:Fork}
We now aim at providing a lower bound on $\Psi(\kappa)$.

\begin{lemma}\label{lem:Au1}
Let $\Phi
\in \Hom^+(\mathcal{A}^{\lambda},\mathbb{R})$ be such that
$\Phi((\beta,1))>0$. If
\begin{equation}\label{eq:forkproof0}
\probability[\boldsymbol{\Phi^{\beta,1}}(\pi^{\beta,2}(\alpha))=c_0]=1\;,
\end{equation}
then $\Phi(\gamma) \ge c_0- \sqrt{\Phi(\Fork)}$.
\end{lemma}
\begin{proof}
By Lemma~\ref{lem:HHKdegree2} applied to
$\Phi^{\wr(\beta,1)}$ we have for every $\epsilon>0$ that
\begin{equation}\label{eq:forkproof01}
\probability[\boldsymbol{(\Phi^{\wr(\beta,1)})^\lambda}(\alpha)<\sqrt{1-\Phi^{\wr(\beta,1)}(\varrho)}+\epsilon]>0\;.
\end{equation}
As in Claim~\ref{cl:sp}, let $\omega=\alpha+\overline{\alpha}+\gamma$  be the formal sum of the three elements of $\calF^\lambda_2$, and
let $\bar\varrho  \in\mathcal{F}_2$ be the edgeless digraph on two vertices. Note that $\bar\varrho = 1 - \varrho$ holds in $\mathcal{A}$ but
$\pi^{\wr(\beta,1)}(\bar{\varrho})=\Fork \neq \pi^{\wr(\beta,1)}(1 -\varrho)$.  The inequality (\ref{eq:forkproof01}) can be rewritten as
$$
\probability[\boldsymbol{(\Phi^{\wr(\beta,1)})^\lambda}\left(\alpha-\left(\sqrt{\Phi^{\wr(\beta,1)}(\bar \varrho)}+\epsilon\right)\omega\right)<0]>0\;.$$ By the counterpositive
of Theorem~\ref{thm:piveeproperties}~b) we have
\begin{equation}\label{eq:forkproof02}
\probability[\boldsymbol{\Phi^{\beta,1}}\left(\pi^{\wr(\beta,1)}_\lambda(\alpha)-\sqrt{\Phi^{\wr(\beta,1)}(\bar\varrho)}\pi^{\wr(\beta,1)}_\lambda(\omega)\right) - \epsilon<0]>0\;
\end{equation}
for every $\epsilon>0$.
By~(\ref{eq-defwr}) we have $\Phi^{\wr(\beta,1)}(\bar \varrho)=\Phi(\Fork)/\Phi(\alpha)^2$, as $\pi^{\wr(\beta,1)}(\bar{\varrho})=\Fork$. 
Theorem~\ref{thm:piproperties}~a) and
the identity $\pi^{\wr(\beta,1)}_\lambda(\omega)=\pi^{\beta,1}(\alpha)$ yield that
$P[\boldsymbol{\Phi^{\beta,1}}(\pi^{\wr(\beta,1)}_\lambda(\omega)) = \Phi(\alpha)]=1$.
We now get from (\ref{eq:forkproof02}) that
\begin{equation}\label{eq:forkproof1}
\probability[\boldsymbol{\Phi^{\beta,1}}(\pi^{\wr(\beta,1)}_{\lambda}(\alpha))<\sqrt{\Phi(\Fork)}+\epsilon]>0\;, 
\end{equation} 
for every $\epsilon>0$.
It
follows from (\ref{eq:forkproof0}) and (\ref{eq:forkproof1}) that
\begin{equation}\label{eq:forkproof2}
\probability[\boldsymbol{\Phi^{\beta,1}}(\pi^{\beta,2}(\alpha)-\pi^{\wr(\beta,1)}_\lambda(\alpha))>
c_0- \sqrt{\Phi(\Fork)}-\epsilon]>0\;.
\end{equation}
We have $\pi^{\beta,2}(\alpha)-\pi^{\wr(\beta,1)}_\lambda(\alpha)
\leq_\beta \pi^{\beta,1}(\gamma)$ (see Figure~\ref{fig:threefurther}).
\begin{figure}
  \centering
\includegraphics[scale=0.92]{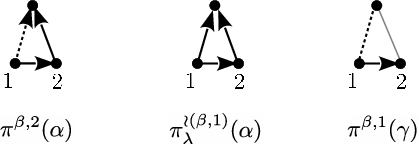}
\caption{Graphs involved in Lemma~\ref{lem:Au1}.}
  \label{fig:threefurther}
\end{figure}
Thus,
\begin{equation}\label{eq:forkproof3}
\probability[\boldsymbol{\Phi^{\beta,1}}(\pi^{\beta,1}(\gamma))>
c_0- \sqrt{\Phi(\Fork)}-\epsilon]>0
\end{equation}
for every $\epsilon>0$.

We claim that $\Phi(\gamma) \ge
c_0- \sqrt{\Phi(\Fork)}$. Indeed, suppose for contradiction
that $\Phi(\gamma)=
c_0- \sqrt{\Phi(\Fork)}-\epsilon'$ for some $\epsilon'>0$.
By Theorem~\ref{thm:piproperties}~a) we conclude that
$$\probability\left[\boldsymbol{\Phi^{\beta,1}}\left(\pi^{\beta,1}(\gamma)- c_0+ \sqrt{\Phi(\Fork)}+\epsilon'\right)=0\right]=1\;,$$ a contradiction to~\eqref{eq:forkproof3}.
The proof of the claim is now finished.
\end{proof}

\begin{lemma}\label{lem:forkdensitylimit} It holds that $\Psi(\kappa)\ge 3(3c_0-1)^2\;$.
\end{lemma}
\begin{proof} 
The homomorphism $\boldsymbol{\Psi^{\lambda}}$
satisfies
$\boldsymbol{\Psi^{\lambda}}((\beta,1))>0$ and (\ref{eq:forkproof0}) with probability one.
Therefore, Lemma~\ref{lem:Au1} implies that
$$\probability\left[\boldsymbol{\Psi^{\lambda}}(\gamma)+\sqrt{\boldsymbol{\Psi^{\lambda}}(\Fork)}\geq c_0\right]=1\;.$$
Consequently,
$$\expectation\left[\sqrt{\boldsymbol{\Psi^{\lambda}}(\Fork)}\right] \geq
c_0-\expectation[\boldsymbol{\Psi^{\lambda}}(\gamma)]=
c_0-\left(1-\expectation[\boldsymbol{\Psi^{\lambda}}(\alpha)]-\expectation[\boldsymbol{\Psi^{\lambda}}(\bar\alpha)]\right)=
3c_0-1\;.$$
The lemma now follows from Jensen's inequality since $\Psi(\kappa)=3\expectation[\boldsymbol{\Psi^{\lambda}}(\Fork)]$.
\end{proof}

We get by expanding $\kappa-3(3c_0-1)^2$ that
\begin{align*}
4\Psi(\kappa-3(3c_0-1)^2) = &\Psi_4+\Psi_7+3\Psi_8+\Psi_{12}+\Psi_{17}+\Psi_{19}+2\Psi_{20}+2\Psi_{21} \notag
\\ &+\Psi_{23}+\Psi_{25}+\Psi_{26}+\Psi_{29}+2\Psi_{30} -
12(3c_0-1)^2\sum_{i=0}^{31}\Psi_i\;.
\end{align*}
Let us write $\FFork(\overline{\Psi})$ for the right hand side of this identity.
The next inequality follows directly from Lemma~\ref{lem:forkdensitylimit}.
\begin{cor}\label{cor:fork} $\FFork(\overline{\Psi}) \geq 0$.
\end{cor}

\subsection{Combining the ingredients}

Let $R(c)$ denote the set of vectors $\overline{r} \in \mathbb{R}^{32}$ such that
\begin{enumerate}
\item[(P1)] $\overline{r} \geq \overline{0}$,
\item[(P2)] $\|\overline{r}\|_1 = 1$,
\item[(P3)] $\overline{a}(A_C(\overline{r}))\overline{a}^{\top} \geq 0$
for every $\overline{a} \in \mathbb{R}^8$,
\item[(P4)] $\overline{b}(B_{R}-cA_{R})\overline{r}^{\top} \geq
0$ for every $\overline{b} \in \mathbb{R}^{14}$,
\item[(P5)] $\Ind_T(\overline{r}) \geq 0$ and $\Ind_V(\overline{r})
\geq 0$, and
\item[(P6)] $\FFork(\overline{r}) \geq 0$.
\end{enumerate}

Corollaries~\ref{cor:CauchySchwarz},~\ref{cor:outreg},~\ref{cor:induction} and~\ref{cor:fork} imply that $\overline{\Psi} \in R(c_0)$.
Therefore, Theorem~\ref{prop:asymptotic} is implied by the next proposition.

\begin{proposition}\label{prop:numeric}
If $c\geq 0.3465$, then $R(c)$ is empty.
\end{proposition}

\begin{proof}
Let
$$\begin{array}{lrrrrrrrrl}
\overline{a}_1 = ( & -69.83, & -27.04, & 3.45, & -53.59, &
1.74, & 28.78, & -9.28, & 59.66 & )\;,\\
\overline{a}_2 = ( & -44.57, & -25.93, & -24.40, & -30.16, &
2.40, & 5.40, & 15.67, & 37.27 & )\;,
\\ \overline{a}_3 = ( & 86.95, & 58.70, & 35.15, & 52.46,&
-18.52, & 3.32, & -52.56, & -57.83 & )\;,
\\ \overline{a}_4 = (
& -1.29, & 0.17, & 57.48, & -26.29, & 10.28, & 26.90, &
-27.33, & -9.15 & )\;,
\end{array}
$$
$$
\overline{b} =
(0,0,-17448,-16496,26501,-24163,-8929,-54193,-30136,7267,-24582,-42769,22644,0),
$$
\begin{align*}
c_T &= 39648\;,\\
c_V &= 19877\;\mbox{, and}\\
d &= 2078\;.
\end{align*}
Further, let
\begin{align*}
F(c,\overline{r}):= \left(\sum_{i=1}^{4}\overline{a_i}(A_C(\overline{r}))\overline{a_i}^{\top}\right) + \overline{b}(B_{R}-cA_{R})\overline{r}^{\top}
&+ c_T\Ind_T(\overline{r}) + c_V \Ind_V(\overline{r}) + d \cdot
\FFork(\overline{r})\;.
\end{align*}
By definition of $R(c)$, we have $F(c,\overline{r}) \geq 0$ for every $\overline{r} \in R(c)$. Moreover, it can be directly verified
(see Maple worksheet, available as an ancillary file on the arXiv)
that
the function $F(c,\overline{r})$ is a non-increasing function for $c\in[1/3,1]$ if $\overline{r} \geq \overline{0}$ is fixed.
On the other hand, we get by computing $F(0.3465,\overline{r})$
\begin{align}\label{eq:final}
\begin{array}{crcrcrcr}
- & 38.906394\:\overline{r}_0 & - & 25.96859\:\overline{r}_1 & - &4156.34069\:\overline{r}_2 & - & 16.447994\:\overline{r}_3 \\
- & 1172.27439\:\overline{r}_4 &- & 577.3814\:\overline{r}_5 & - & 4.57689\:\overline{r}_6 & - &10.55419\:\overline{r}_7 \\
- & 4042.1489\:\overline{r}_8 & - & 10.328894\:\overline{r}_9& - & 13.03079\:\overline{r}_{10} & - & 1327.03609\:\overline{r}_{11}\\
- &2658.54869\:\overline{r}_{12} & - & 9.71489\:\overline{r}_{13} & - & 14574.68439\:\overline{r}_{14}& - & 7.032994\:\overline{r}_{15}\\
- & 6.85949\:\overline{r}_{16} & - &11279.04479\:\overline{r}_{17} & - & 7.458494\:\overline{r}_{18} & - & 15538.64129\:\overline{r}_{19}\\
- & 19.61149\:\overline{r}_{20} & - & 15.87099\:\overline{r}_{21} & - &12.39949\:\overline{r}_{22} & - & 9949.057894\:\overline{r}_{23}\\
- & 9.5492\:\overline{r}_{24} &- & 12.55709400\:\overline{r}_{25} & - & 17.24429\:\overline{r}_{26} & - &9.535194\:\overline{r}_{27}\\
- & 1.24639\:\overline{r}_{28} & - & 3070.47399\:\overline{r}_{29}&
- & 17.36519\:\overline{r}_{30} & - & 13.03819\:\overline{r}_{31}\;.
\end{array}
\end{align}
The coefficients in (\ref{eq:final}) are exact.
The expression (\ref{eq:final}) is bounded from above by
$F(0.3465,\overline{r})\le-1.24639\|\overline{r}\|_1<0$.
Proposition~\ref{prop:numeric} follows.
\end{proof}

The proof of the proposition concludes the section since Theorem~\ref{prop:asymptotic} is now proven.

\section{Concluding remarks}

\subsection{The algorithm}\label{sec:algorithm} We now
sketch how a computer aided us with a search for the right coefficients used in the estimates.
For a fixed value of $c$, the conditions (P1), (P2) and (P4)--(P6) in the definition of $R(c)$
give a polytope in $\mathbb{R}^{32}$, which is further denoted by $R'(c)$.
The polytope $R'(c)$ is defined by $68$ linear identities and inequalities.
It is not hard to check (using linear programming) whether $R'(c)$ is empty,
while it is less straightforward to verify whether $R'(c)$ contains a point that
satisfies the condition (P3) for every $\overline{a} \in \mathbb{R}^{8}$.
Since a particular point $\overline{r} \in \mathbb{R}^{32}$ satisfies (P3) for all choices $\overline{a}$ iff $A_C(\overline{r})$ is positive semidefinite,
the search problem for the value $c$ when $R(c)$ becomes empty can be expressed as a semidefinite program.
However, we have searched for such value $c$ using an iterative linear programming based algorithm,
which required less programming.

\subsection{Improving the bound}

It is possible to improve the bound we give in
Theorem~\ref{thm:main} slightly at the expense of a more
involved proof. As the improvements we were able to produce
were not significant and relied on technical and
computational tweaks, rather than new ideas.
So, we have chosen to present a simpler bound.
Still, let us point out some ways to improve the bound we give.

{\bf 1.} Instead of concentrating on the values $\Psi$
takes on elements of $\mathcal{F}_{4}$, one can examine
$\mathcal{F}_k$ for $k=5,6,\ldots$. This leads to a significant increase of matrices $A_C, A_{R}$ and $B_{R}$ and
the involved coefficients in size. R\'emi de Joannis de Verclos, Jean-S\'ebastien Sereni and Jan Volec pushed the bound in Theorem~\ref{thm:main} down to~0.3388 in March 2014
by extending the approach presented in this paper. (Their proof uses elements $\mathcal{F}_6$ in calculations and it incorporates several additional bounds.)
It seems likely that the bound can be pushed even further.

{\bf 2.} One can attempt to generalize
Lemma~\ref{lem:induction}.
We initially considered a more general class of induction hypotheses 
but it turned out that adding additional estimates does not lead to
further improvements of the bound.

{\bf 3.} Chudnovsky, Seymour and Sullivan~\cite{ChuSeySull}
conjectured that one can delete $k/2$ edges from a
triangle-free digraph $D$ with at most $k$ non-edges to
make it acyclic. This conjecture implies improvements of
Lemmas~\ref{lem:HHKdegree} and~\ref{lem:HHKdegree2}.
Dunkum, Hamburger and P\'or~\cite{DunHamPor} have recently
shown that deleting $0.88k$ edges suffices, and
Shen [private communication] further improved this bound to
$0.865k$. These results allow us to improve
Lemma~\ref{lem:HHKdegree2}, but such an improvement in turn
only produces a tiny decrease in the bound
Theorem~\ref{thm:main}. However, the proofs
in~\cite{ChuSeySull,DunHamPor} can both be recast in the
language of flag algebras. It might be  interesting see
if one can obtain a generalization of
Lemma~\ref{lem:HHKdegree} in this manner and use it to
improve Theorem~\ref{thm:main}.

\section*{Acknowledgement}
The computer search for the coefficients of the vector $F(c_0,\overline r)$ was done using QSopt linear programming solver.
We would like to thank its author William Cook for making the tool available and answering our questions concerning its usage.
We are also grateful for valuable comments on the manuscript provided by the anonymous referees. Finally, we would like to thank
Jean-S\'ebastien Sereni for his extremely detailed comments including pointing out typos in some of our formulas.

\bibliographystyle{plain}
\bibliography{bibl}
\end{document}